\documentclass[leqno]{article}
\usepackage{amsmath}
\usepackage{xcolor}
\usepackage{amssymb}
\usepackage{amsthm}                                         
\usepackage{url}
\usepackage{graphicx}
\usepackage{epsfig}
\newtheorem{theorem}{\bf Theorem}

\newtheorem{corollary}[theorem]{\bf Corollary}
\newtheorem{lemma}[theorem]{\bf Lemma}
\newtheorem{afirmacion}[theorem]{\bf Claim}
\newtheorem{proposition}[theorem]{\bf Proposition}
\newtheorem{definition}[theorem]{\bf Definition}

\newtheorem{remark}[theorem]{\bf Remark}

\numberwithin{equation}{section}
\numberwithin{theorem}{section}
\numberwithin{figure}{section}

\def\p{\shortmid}
\def\pp{\shortparallel}
\def\R{\mathbb{R}}

\def\R{\mathbb{R}}
\def\CL{\mathcal{L}}

\def\CV{\mathcal{V}}
\def\CA{\mathcal{A}}
\begin{document}
\renewcommand{\thefootnote}{}
\footnotetext{The authors were partially supported by MICINN-FEDER, Grant No. MTM2016- 80313-P and  Junta de Andaluc\'ia Grant No. FQM325 }

\title{Equilibrium of Surfaces in a Vertical  Force Field} 
\author{Antonio Mart\'{\i}nez and A.L. Mart\'inez-Trivi\~no}
\vspace{.1in}

\date{}
\maketitle
\noindent {\footnotesize Department of Geometry and Topology, University of Granada, E-18071 Granada, Spain.\\ 
e-mails: amartine@ugr.es, aluismartinez@ugr.es}
\begin{abstract}
In this paper we study $\varphi$-minimal surfaces in $\R^3$ when the function $\varphi$ is invariant under a two-parametric group of translations. Particularly those which are complete graphs over domains in $\R^2$. We describe a full classification of complete flat  embedded  $\varphi$-minimal surfaces if  $\varphi$ is strictly monotone and characterize rotational $\varphi$-minimal surfaces
 by its behavior at infinity when $\varphi$ has a quadratic growth.
\end{abstract}


\everymath={\displaystyle}
\section{Introduction}
The equilibrium  of a flexible, inextensible surface  $\Sigma$  in a force field ${\cal F}=(X,Y,Z)$ of $\R^3$, was given by Poisson \cite[pp. 173-187]{Poisson} and   when the intrinsic forces of the surface  are assumed to be equal, the external force must have a potential $ {\cal T}$  which corresponds, up to a constant, with the tension of the surface, that is, 
\begin{equation}
d{\cal T} + X dx + Y dy + Z dz = 0.\label{0fminimal}
\end{equation} In this case, the equilibrium condition is given  in terms of the mean curvature vector  ${\text {\bf H}}$ of $\Sigma$ as follows: 
\begin{align} 
&{\text {\bf H}} {\cal T}+ {\cal F}^\perp = 0\label{fminimal}
\end{align}
where  $\perp$ denotes the projection to the normal bundle of  $\Sigma$.

\

From equations \eqref{0fminimal} and \eqref{fminimal} Poisson obtains:
\begin{itemize}
\item {\sl The minimal surface equation},  by taking ${\cal F}=0$ and ${\cal T}=const.$.
\item  {\sl The capillary surface equation},  by taking  ${\cal T}={\text const}$ and ${\cal F}$ normal to the surface with $\|{\cal F}\|$ depending linearly on the height.
\item  {\sl The equation of a heavy surface in a gravitational field},  by taking  ${\cal T}=(0,0,g\,{\cal E}(z))$, $g$ = gravitational constant and ${\cal E}(z)$ a density function on  the surface.
\end{itemize}

In this paper we are interested in the last case, that is, when  the equation \eqref{fminimal}  gives
\begin{equation}
{\text {\bf H}} = (\overline{\nabla} \varphi)^\perp = \dot{\varphi} \ \vec{e}_3^{\,\perp}, \label{vforce}
\end{equation}
where $\varphi (z) = \log\int^z _{z_0}g\,{\cal E}(t))dt$, $\overline{\nabla}$ is the gradient operator in $\R^3$ and  $(\ \dot{ }\ )$  denotes derivate respect to the third coordinate. To get a regular problem we have to restrict the surfaces to the region of $\R^3$ where $\varphi$ is  regular. 
These surfaces are a particular case of the so called $f$-minimal surfaces (see \cite{CMZ}) for which the function $f$ depends only on the height. They   can be   viewed either as critical points of the weighted volume functional
\begin{equation}
\label{critiarea}
V_{\varphi}(\Sigma):=\int_{\Sigma}e^{\varphi}\, dA_{\Sigma},
\end{equation}
where $dA_{\Sigma}$ is the volume element of $\Sigma$, or as  minimal surfaces in $\R^3$ with  the conformally changed metric
\begin{equation} 
\label{ilm}
G_\varphi:=  \mathrm{e}^\varphi\ \langle \cdot , \cdot\rangle.
\end{equation}
From this property of minimality, a tangency principle can be applied and  any two different $\varphi$-minimal  surfaces cannot ``touch" each other at one interior or boundary point (see \cite[Theorem 1 and Theorem 1a]{E}).

\

\renewcommand{\theequation}{\arabic{equation}}
\numberwithin{equation}{section}
Any surface satisfying $\eqref{vforce}$ will be called  {\sl $[\varphi, \vec{e}_3]$-minimal } 
and if  $\Sigma$ is   the vertical graph of a function $u:\Omega\subseteq \R^2 \longrightarrow \R$, we  also refer to $u$ as  $[\varphi, \vec{e}_3]$-minimal. Hence, $u$ is  $[\varphi, \vec{e}_3]$-minimal if and only if it solves the following $[\varphi, \vec{e}_3]$-minimal equation:\begin{equation}
(1+u_{x}^2)u_{yy} + (1+u_{y}^2)u_{xx} - 2u_{y}u_{x}u_{xy}= \dot{\varphi}(u)\left( 1+u_{x}^2+u_{y}^{2} \right).  \label{fe}
\end{equation} 

\

This kind of surfaces has been widely studied specially from the viewpoint of calculus of variations. Classical results about the Euler equation and the existence and regularity for  the solutions of the Plateau problem for \eqref{critiarea} can be found in \cite{BHT, H1, H2, HK, T}.

But contributions from a more geometric viewpoint only has been given for some  particular  functions $\varphi$. It is interesting to mention 
\begin{itemize}
\item The case of $\varphi(z) = z$:  it corresponds with translating solitons, that is, surfaces  in $\R^3$ such that $$ t\mapsto \Sigma + t \vec{e}_3$$ is a mean curvature flow, i. e. such that normal component of the velocity at each point is equal to the mean curvature at that point:
$ {\text {\bf H}} =  \vec{e}_3^{\,\perp}$. Recent advances in the understanding of its local and global geometry  can be found in \cite{CSS, JS, HMW, HIMW, HIMW2,MSHS1, MSHS2, SX, W}
\item  The case of $\varphi(z) =\alpha \log z$, $\alpha$=const. It includes the two dimensional examples analogues of the catenaries (when $\alpha =1$). We refer to \cite{BHT, D, DH, N, Rafa} for some  progress made in this family.
\end{itemize}

The aim of this paper is develop a general  and systematic approach to study 
 $[\varphi, \vec{e}_3]$-minimal surfaces from a geometric viewpoint.
Nonetheless, the class of $[\varphi, \vec{e}_3]$-minimal surfaces  is indeed very large and much richer in whats refers to examples and geometric behaviors. Although  new ideas are needed for its study, it will be necessary, in order to get classification results,  to impose some additional conditions to the function $\varphi$. Here, as a general assumption we will always consider $\varphi$ strictly monotone, that is,
\begin{align}\label{hip}
& \varphi:]a,b[\subseteq \R \rightarrow \mathbb{R} \text{ is a strictly increasing (or decreasing) function}\\
& \text{and  $\Sigma\subset \R^2 \times ]a,b[$.}\nonumber
\end{align}

\
 
 Invariant surfaces by an uniparametric group of rigid motions in $\R^3$ are related with the one dimensional case of \eqref{fe}. Since $\varphi$ is taking  so arbitrary, we   only   consider $[\varphi, \vec{e}_3]$-minimal surfaces invariant by two types of uniparametric groups, namely,  groups of horizontal translations and the group of vertical rotations.
 
 In the first case, besides vertical planes, we may consider that $u=u(x)$, $x\in I$ depends  only on  $x$. Then, from \eqref{fe}, the  generalized cylinder $\Sigma =\{ (x, y, u(x)) \ | \ x\in I , y\in \R\}$ is a $[\varphi, \vec{e}_3]$-minimal surface if and only if $u$ satisfies
\begin{align}
u''(x)=\dot{\varphi}(u)(1+u'(x)^{2}) \label{he}
\end{align}
From its physical interpretation, any solution of \eqref{he} will be called  {\sl $\varphi$-catenary}. The corresponding generalized cylinder is called {\sl $[\varphi, \vec{e}_3]$-catenary cylinder}. If we rotate around the $x$-axis a $[\varphi, \vec{e}_3]$-catenary cylinder an angle $\theta\in ]0,\pi/2[$ and dilate by $\frac{1}{\cos \theta}$, the resulting surface is also $[\varphi, \vec{e}_3]$-minimal  and we will say it is a  {\sl tilted $[\varphi, \vec{e}_3]$-catenary cylinder}. In Theorem 3.7 we prove that any complete flat $[\varphi, \vec{e}_3]$-minimal  surface is either a vertical plane or a $[\varphi,\vec{e}_3]$-catenary cylinder (maybe tilted). 

In the second case we  consider   $[\varphi,\vec{e}_3]$-minimal surfaces that  are invariant under  the one-parameter group of rotations that fix the $\vec{e}_{3}$ direction. From \eqref{fe}, the arc-lenght parametrized generating curve  $$\gamma(s)=(x(s), 0, z(s)), \qquad   \ \ s\in I\subset\mathbb{R}$$ of a such surface satisfies
\begin{equation}
\label{gcurve}
\left\{ \begin{array}{l}
x'=\cos(\theta)\\
z'=\sin(\theta),\\
\theta'=\dot{\varphi}(z)\text{cos}(\theta)-\frac{\text{sin}(\theta)}{x}.
\end{array}\right. 
\end{equation}
In Theorems \ref{bowl} and \ref{existencecatenoid}, we establish the geometric properties of the rotational  $[\varphi,\vec{e}_3]$-minimal surfaces according two types of surfaces: one is globally convex  with only one complete embedded end (it is called a {\sl $[\varphi,\vec{e}_3]$-minimal bowl}) and the other has two complete embedded convex ends and  has a generating curve of winglike type (it is called {\sl $[\varphi,\vec{e}_3]$-minimal catenoid})

\

Very little is known  about the  geometry of the immersed $[\varphi,\vec{e}_3]$-minimal surfaces and most of the results have been proved only for translating solitons. One of the first result in that direction was obtained by  Clutterbuck, Schn\"ure, Schulze in \cite{CSS}, where they proved that when $\dot{\varphi}\equiv 1$, any rotationally symmetric solution $u=u(r)$, $r=\sqrt{x^2+y^2}$, on the exterior of a compact planar domain has de following asymptotic behaviour: 
$$ u(r) = \frac{r^2}{2} - \log r + O(r^{-1}).$$
Somewhat later Martin-Savas-Smoczyk proved in \cite{MSHS2} that any complete translating soliton with a single end asymptotic to a translating paraboloid is a translating paraboloid.
\\

In this paper we generalize the above results to $[\varphi,\vec{e}_3]$-minimal  with $\dot{\varphi}$ satisfying  the following expansion at infinity
\begin{align}\label{series2}
\dot{\varphi}(u) = \alpha u + \beta + \sum_{n=1}^\infty\frac{a_n}{u^n}, \quad a_n\in \R,
\end{align}
where either $\alpha>0$ and the first non-vanishing $a_k$ is positive or $\alpha=0$, $\beta>0$ and the first non-vanishing $a_k$ is negative. The results we prove  can be summarized in the following two theorems
\renewcommand\thetheorem{\Alph{theorem}}
%\numberwithin{equation}{section}
\begin{theorem} If $\dot{\varphi}$ satisfies \eqref{series2}, then any  rotationally symmetric solution $u$ of  \eqref{fe} has  the following asymptotic behavior,
\begin{itemize}
\item  If  $\alpha>0$, 
\begin{equation}\label{applineal}
\varphi(u)(r)= C \ e^{\alpha \,r^{2}}  + O(r^2), \quad C>0,
\end{equation}
\item If $\alpha=0$ and  up to a constant, we have,
\begin{equation}
\label{casoalphacero}
{\cal G}(u)(r)=\frac{r^{2}}{2}-\frac{1}{\beta^2}\log(r)+ {O}(r^{-2}),
\end{equation}
 where ${\cal G}$ is the strictly increasing function given by ${\cal G}(u)=\int_{u_0}^u\frac{d\xi}{\dot{\varphi}(\xi)}$.
 \end{itemize}
\end{theorem}
\begin{theorem}
\label{unicidad}
Let  $\Sigma$ be a complete properly embedded  $[\varphi,\vec{e}_3]$-minimal surface in $\mathbb{R}^{3}$ with a single end that is smoothly asymptotic to a  $[\varphi,\vec{e}_3]$-minimal bowl, $\dot{\varphi}$ satisfying  \eqref{series2}. Then the surface $\Sigma$ is a  $[\varphi,\vec{e}_3]$-minimal bowl.
\end{theorem}
\renewcommand{\thetheorem}{\arabic{equation}}
\numberwithin{theorem}{section}
The paper is organized as follows, in  Section  2 we show some fundamental equations related to our family of surfaces and as a consequence we prove the non-existence  of  closed examples and two results about strictly convexity and mean convexity of  $[\varphi, \vec{e}_3]$-minimal surfaces. 

Section 3 is devoted to the study and classification of embedded complete flat $[\varphi,\vec{e}_{3}]$-minimal surfaces. We describe geometrically the so called  $[\varphi,\vec{e}_{3}]$-catenary cylinders and tilted $[\varphi,\vec{e}_{3}]$-catenary cylinders and characterize  them together to vertical planes as the unique examples of complete flat  $[\varphi,\vec{e}_{3}]$-minimal surfaces.

In Section 4 we study  the existence and classification of rotational examples. We construct for $\varphi$ in  a very general class of functions (strictly increasing and convex) a family of $[\varphi,\vec{e}_{3}]$-minimal bowls (which are strictly convex graphs) and $[\varphi,\vec{e}_{3}]$-minimal catenoids with a winglike shape (which resemble the usual translating catenoids in $\R^3$).

Finally, Sections 5 and 6 are devoted  to study  $[\varphi,\vec{e}_{3}]$-minimal surfaces when $\varphi$ has a quadratic growth. We provide the  asymptotic behavior of rotationally symmetric  examples and characterize $[\varphi,\vec{e}_{3}]$-minimal bowls by their behavior at infinity.

\

{\bf Acknowledgements}: The
authors are grateful to Margarita Arias, Jos\'e Antonio G\'alvez and Francisco Mart\'in  for helpful comments during the preparation of this manuscript.

\section{Some relevant equations}
Here, we will give some local fundamental equations related to $[\varphi,\vec{e}_3]$-minimal surfaces.  Let $\psi:M \longrightarrow\mathbb{R}^{3}$ be a  $2$-dimensional $[\varphi,\vec{e}_3]$-minimal immersion (maybe with a non empty boundary)  with Gauss map $N$,  induced metric $g$ and second fundamental form ${\rm {\bf A}}$. We shall denote by $\nabla$, $\Delta$ and $\nabla^2$, respectively,  the Gradient, Laplacian and Hessian  operators of $g$.  
\\

The mean curvature vector of $\psi$ is defined by ${\text {\bf H}} = {\text {\rm trace}}_g{\text {\bf A}}$ and the  symmetric bilinear form $	\CA$ given by $\CA(X,Y) = -\langle{\text {\bf A}}(X,Y),N\rangle$, $X,Y\in T\Sigma$, is called  scalar second fundamental form. The mean curvature function $H$ will be the trace of $\CA$ with respect to $g$. With this notation, \eqref{vforce} is equivalent to 
\begin{equation}
H:= - \dot{\varphi} \langle N,\vec{e}_3\rangle.\label{me}
\end{equation}
We will assume that $\varphi$ satisfies \eqref{hip} and let us introduce the height  and angle functions, respectively,  by:
$$ \mu:= \langle\psi,\vec{e}_3\rangle, \quad \eta:= \langle N,\vec{e}_3\rangle.$$ 
\begin{lemma}
\label{fl} 
 The following relations hold
\renewcommand{\theequation}{\arabic{equation}}
\setcounter{equation}{0}
\begin{align}
&\nabla \mu=\vec{e}_3^{\top},\qquad \langle\nabla\eta,\, \cdot\,\rangle=\CA(\nabla \mu,\, \cdot\,),\label{e1}\\
&\dot{\varphi}^{2}=\dot{\varphi}^{2}\vert\nabla  \mu \vert^{2}+H ^{2},\label{e2}\\
&\dot{\varphi}\nabla^{2} \mu =H \CA,\label{e3}\\
&\nabla^{2}\eta=(\nabla\CA)(\nabla  \mu,\, \cdot\,,\, \cdot\,)  +\frac{H}{\dot{\varphi}}\CA^{[2]},\label{e4}\\
&\Delta  \mu=\dot{\varphi}(1-\vert\nabla  \mu\vert^{2}),\label{e5}\\
&\Delta N + \dot{\varphi}\nabla\eta + \ddot{\varphi} \eta \nabla \mu +\vert \CA\vert^{2} N=0,\label{e6}\\
&\nabla^{2}H =-\eta\nabla^2\dot{\varphi} - (\nabla\CA)(\nabla\varphi,\,\cdot\,,\, \cdot\,)- H \CA^{[2]}+ {\cal B}\label{e7}\\ 
&\Delta \CA + (\nabla\CA)(\nabla\varphi,\, \cdot\,,\, \cdot\,)+\eta\nabla^2\dot{\varphi} + |\CA|^2\CA - {\cal B}=0,\label{e8}
\end{align}
where   $\CA^{[2]}$ and ${\cal B}$ are the symmetric 2-tensors given by the following expressions: 
\begin{align*}
&\CA^{[2]}(X,Y)=\sum_{k}\CA(X,E_{k})\CA(E_{k},Y),\\
&{\cal B}(X,Y)=\langle\nabla \dot{\varphi},X,\rangle \CA(\nabla \mu,Y) +\langle\nabla \dot{\varphi},Y\rangle \CA(\nabla \mu,X),
\end{align*}
for any vector fields $X,Y\in T\Sigma$ and any  orthonormal frame $\{E_{1},E_2\}$ of $T\Sigma$.
\end{lemma}
\begin{proof}
\begin{enumerate}
\item[(1)]  Differentiating $\mu$ and $\eta$ respect to any $X\in T\Sigma$, we get, 
\begin{align*}
&\langle\nabla  \mu,X\rangle=d \mu(X)=\langle \vec{e}_{3}^{\top},X\rangle, \\
&\langle \nabla\eta,X\rangle =d\eta(X)=\langle dN(X),\vec{e}_{3}^{\top}\rangle=\CA(X, \vec{e}_{3}^{\top}).
\end{align*}
\item[(2)] From \eqref{me} and  \eqref{e1}, it is clear that
$$1=\vert \nabla  \mu\vert^{2}+\frac{H ^{2}}{\dot{\varphi}^{2}}.$$
\item[(3)] From definition of the Hessian operator, 
$$\nabla^{2} \mu(X,Y)=XY( \mu)-(\nabla_{X}Y)( \mu)=\langle {\text{\bf A}}(X,Y),e_{3}\rangle = -\CA(X,Y)\eta.$$
So \eqref{e3} follows from  \eqref{me}.
\item[(4)] From Codazzi equation and \eqref{me}:
\begin{align*}
&\nabla^{2}\eta(X,Y)= \sum_{k}(\nabla\CA)(E_k,X,Y)E_{k}(\mu)-\sum_{k}\CA(X,E_k)\CA(Y,E_k)\eta =\\
&=(\nabla\CA )(\nabla \mu, X,Y)+\frac{H }{\dot{\varphi}}A^{[2]}(X,Y).
\end{align*}
\item[(5)] From \eqref{e2} and \eqref{e3},
$$ \Delta \mu = \sum_{k}\nabla^2\mu(E_k,E_k)=\frac{H^2}{\dot{\varphi}}=\dot{\varphi}(1-\vert\nabla  \mu\vert^{2}).$$
\item[(6)]  As $H=-\dot{\varphi} \eta$, we have $$\nabla H =-\ddot{\varphi}\eta\nabla u-\dot{\varphi}\nabla\eta,$$ and \eqref{e6} follows from the well known fact that  $\Delta N=\nabla H -\vert A\vert^{2}N$.
\item[(7)] From \eqref{me} and \eqref{e4} we obtain
\begin{align*}
&\nabla^{2}H (X,Y)=XY(H )-(D_{X}Y)H = \\ 
&=-\eta\nabla^{2}\dot{\varphi}(X,Y)+\dot{\varphi}\nabla^{2}\eta(X,Y)
+\langle\nabla\dot{\varphi},Y\rangle\langle X,\nabla\eta\rangle + \langle\nabla \dot{\varphi},X\rangle\langle Y,\nabla\eta\rangle=\\
&=-\eta\nabla^2\dot{\varphi}(X,Y) - (\nabla\CA)(\nabla\varphi,X,Y)- H \CA^{[2]}(X,Y)+ {\cal B}(X,Y).
\end{align*}
which give the proof of \eqref{e7}.
\item[(8)] Using the well known Simon's identity:
$$ \Delta \CA = \nabla^2 H - |\CA|^2 \CA + H\CA^{[2]} $$ and  \eqref{e7}  we obtain \eqref{e8}.
\end{enumerate}
\end{proof}
From this Lemma we have,
\begin{corollary}
\label{belowplane}
If  $\varphi:]a,b[\rightarrow\mathbb{R}$, is  a strictly increasing (or decreasing) function, then the height function $\mu$ of $\psi$ cannot attain a local maximum (or local minimum) at  any interior point.
\end{corollary}
\begin{corollary}
There is no any closed $2$-dimensional $[\varphi,\vec{e}_{3}]$-minimal immersion  $ \psi:M\longrightarrow\mathbb{R}^{2}\times ]a,b[$.
\end{corollary}
About the sign of the curvatures of $\psi$ we have,
\begin{theorem}
\label{pmcurvaturamedia}
Let $\varphi:]a,b[\rightarrow\mathbb{R}$ be  a strictly increasing  function satisfying 
\begin{equation} 
\ddot{\varphi}  + \lambda \, \dot{\varphi}^2 \geq 0, \qquad \text{{\rm for some constant  $\lambda>0$}},\label{condition}
\end{equation}
and let $\psi:\Sigma\longrightarrow\mathbb{R}^{2}\times]a,b[$ be a $2$-dimensional  $[\varphi,\vec{e}_3]$-minimal immersion with  $H \leq 0$. If $H $ vanishes anywhere, then $H $  vanishes everywhere and $\psi(\Sigma)$ lies in a vertical plane.
\end{theorem}
\begin{proof}
By using \eqref{me} and the equations \eqref{e1}, \eqref{e2},  \eqref{e5} and  \eqref{e6}  in Lemma \ref{fl}, we have
\begin{align*}
&\Delta  (\mathrm{e}^{-\lambda\varphi})+\lambda\mathrm{e}^{-\lambda\varphi}(\ddot{\varphi}|\nabla\mu|^2 + H^2- \lambda\dot{\varphi}^2|\nabla\mu|^2)=0,\\
& \Delta\eta+\dot{\varphi}\langle\nabla\eta,\nabla \mu\rangle+(\vert A\vert^{2}+\ddot{\varphi}\vert\nabla \mu\vert^{2})\eta=0.
\end{align*}
Thus, we obtain 
\begin{align*}
& \Delta  (\mathrm{e}^{-\lambda\varphi} \eta) + (2 \lambda+ 1) \langle \nabla(\mathrm{e}^{-\lambda\varphi} \eta), \nabla\varphi\rangle = \\
&=- \eta  \mathrm{e}^{-\lambda\varphi}(  (\lambda+1) (\ddot{\varphi} + \lambda\,\dot{\varphi}^2)|\nabla \mu|^2 + \lambda H^2 + |\CA|^2).
\end{align*}
But, by hypothesis, $\eta$ is a  nonnegative function, and so, from the strong maximum principle, if it vanishes anywhere  then it  vanishes everywhere, which  concludes the proof.
\end{proof}

\begin{theorem}
\label{pmcurvaturagauss}
Let $\varphi:]a,b[\rightarrow\mathbb{R}$ be  a strictly increasing  function satisfying  $\dddot{\varphi}\leq0$, 
and let $\psi: \Sigma\longrightarrow\mathbb{R}^{2}\times]a,b[$ be a $2$-dimensional  locally convex $[\varphi,\vec{e}_3]$-minimal immersion. If the Gauss curvature $K$ vanishes anywhere, then $K$  vanishes everywhere.
\end{theorem}
\begin{proof}
By hypothesis, the Gauss map $N$ can be chosen such that  $\CA$ is a positive  semi-definite bilinear form and  from \eqref{e8}, we have
\begin{align*} &\Delta \CA + (\nabla\CA)(\nabla\varphi,\, .\,,\, .\,)+{\cal G}(\CA) =0\end{align*}
where $$ {\cal G}(\CA) = \eta\nabla^2\dot{\varphi} + |\CA|^2\CA - {\cal B}.
$$
But, from Lemma \ref{fl}, if  $\dddot{\varphi}\leq0$ we obtain ${\cal G}(\CA)(v,v) = \eta \dddot{\varphi}\langle\nabla \mu,v\rangle^2\leq 0$  for each null vector $v$ of $\CA$. So, 
can apply the maximum principle of Hamilton (see \cite[Section 2]{SavasSmoczyk}) and if there is an interior point of $\Sigma$  where $\CA$ has a null-eigenvalue then $\CA$ must have a null-eigenvalue everywhere,  which concludes the proof of the theorem.
\end{proof}

\section{Complete flat $[\varphi, \vec{e}_3]$-minimal surfaces}\label{s3}

\subsection{Vertical graphs invariant by horizontal translations}
Consider the   $[\varphi, \vec{e}_3]$-minimal vertical graph given by a function $u$ which  only depend on one variable, $u=u(x)$, from \eqref{fe} $u$ must be  a solution of the following ODE:
\begin{align}
u''(x)=\dot{\varphi}(u)(1+u'(x)^{2}) \label{htrans}
\end{align}
In order to look for complete examples we will consider that 
  $$\varphi:\ ]a,\infty[\ \longrightarrow \R$$ 
 is  either a strictly increasing (or decreasing) function.  Then, by taking $z= \varphi(u)$ and  $u' = \tan(v)$, we obtain that \eqref{htrans}  is equivalent to
\begin{equation}
\left.
\begin{array}{l}
 v' = h(z),\\
 z'= h(z) \tan(v),
\end{array}
\right\}
\label{translation}
\end{equation}
where $h(z)=\dot{\varphi}(\varphi^{-1}(z))$. 

It is clear that 
$e^{z}\text{cos}(v) $ is constant 
along the solutions of \eqref{translation} and from Figure \ref{diagramafase}, for each solution $u$ of \eqref{htrans}  there exists a unique  $x_0 \in \R$ such that $v(x_0)=0$ (it is not a restriction to assume  that $x_0=0$). 

\begin{figure}[h]
\label{diagramafase}
\begin{center}
\includegraphics[width=.6\textwidth]{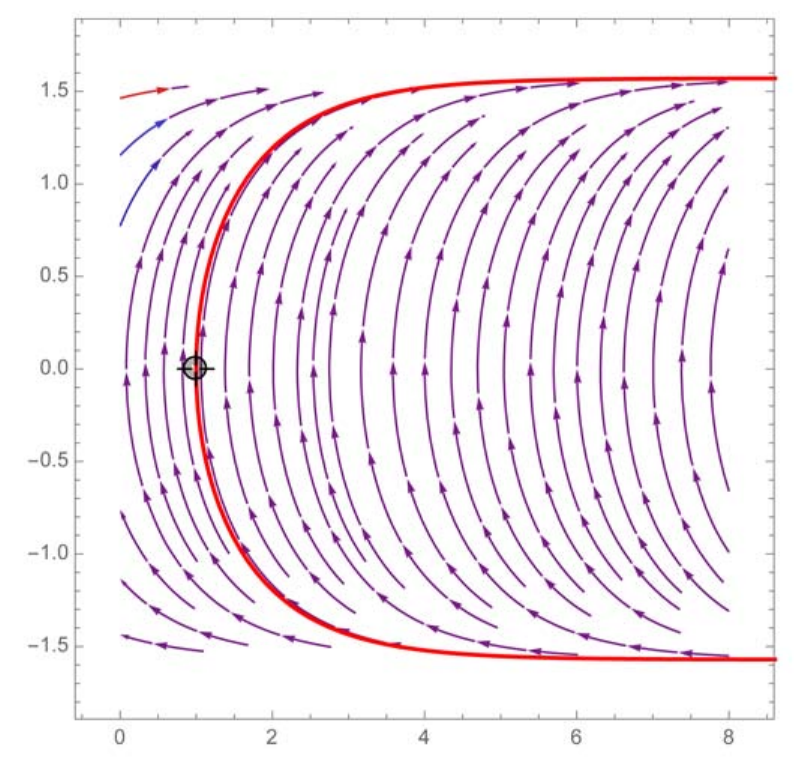}
\end{center}
\caption{Phase portrait of  \eqref{translation}}
\end{figure}

\

By taking the initial conditions 
\begin{equation}
u(0)=u_0, \quad u'(0)=0,\label{ci} 
\end{equation} 
 we have that for each $x\geq 0$,  $u(x)$ is given  by
 \begin{equation}
u(x):= (\mathcal{X}\circ \varphi)^{-1}(x), \quad \text{with}\quad \mathcal{X}(z) = \int_{z_0}^z\frac{d\tau}{|h(\tau)|\sqrt{\mathrm{e}^{2(\tau -z_0)} - 1}},\label{solution}
\end{equation}
where $z_0=\varphi(u_0)$. Thus, 
from \eqref{htrans} and \eqref{ci}, we obtain,
\begin{proposition}\label{par}
 The solution $u$  of \eqref{htrans}-\eqref{ci}   is  even and it is defined in the  interval  $]-\Lambda_{u_0},\Lambda_{u_0}[$,  where 
\begin{equation}
\Lambda_{u_0} = \lim_{u\rightarrow \infty} \int_{\varphi(u_0)}^{\varphi(u)}\frac{d\tau}{|h(\tau)|\sqrt{\mathrm{e}^{2(\tau -z_0)} - 1}}.\label{Lambda}
\end{equation}
\end{proposition}
\begin{theorem}\label{3-2} If  $\varphi:\ ]a,\infty[\ \longrightarrow \R$ is a strictly increasing function, then,
\begin{itemize}
\item   $\Lambda_{u_0}<\infty$ if and only if $\int_{u_0}^\infty \mathrm{e}^{-\varphi(\lambda)} d\lambda < \infty $. So, if   $\Lambda_{\lambda_0}<\infty$ for some $\lambda_0\in]a,\infty[$,  then $\Lambda_\lambda<\infty$ for all $\lambda\in ]a,\infty[$.
\item  If   $\Lambda_{\lambda}<\infty$ and $\dot{\varphi}$ is increasing (respectively, decreasing), then $\Lambda_\lambda$ is  decreasing (respectively, increasing) in $\lambda$.
\end{itemize}
\end{theorem}
\begin{proof}
As $$ \lim_{\tau\rightarrow \infty}\frac{\sqrt{\mathrm{e}^{2(\tau -z_0)} - 1}}{\mathrm{e}^{\tau -z_0}} = 1\neq 0,$$  the first item follows from \eqref{Lambda}.

On the other hand, by assuming that $\dot{\varphi}$ is increasing  and  $\Lambda_{\lambda}<\infty$ for all $\lambda\in]a,\infty[$, we have from \eqref{Lambda}, that, if   $\lambda_1\leq \lambda_2$, 
$$  \Lambda_{\lambda_1}\geq \Lambda_{\lambda_2} + \lim_{z\rightarrow \infty} \int_{z-\varphi(\lambda_2)}^{z-\varphi(\lambda_1)} \frac{d\tau}{h(\tau + \varphi(\lambda_1))\sqrt{\mathrm{e}^{2\tau} - 1}} = \Lambda_{\lambda_2}.$$
A similar discussion can be done when $\dot{\varphi}$ is decreasing.
\end{proof}

From \eqref{htrans}, \eqref{translation}, \eqref{ci},  \eqref{solution},\eqref{Lambda} and Theorem \ref{3-2},  we can prove the following properties  of the solutions,
\begin{theorem}\label{t2} Let $\varphi:\ ]a,\infty[\ \longrightarrow \ ]b,c[$,  $ a,b\in \R\cup\{-\infty\}$,  $c\in  \R\cup\{\infty\}$ be a strictly increasing diffeomorphism, then the solution $u$  of \eqref{htrans}-\eqref{ci} is defined in $]-\Lambda_{u_0},\Lambda_{u_0}[$, $\Lambda_{u_0} \in \{\R^+,\infty\}$, it is convex,  symmetric about the $y$-axis and has a minimum at $x=0$. Moreover, 
\begin{itemize}
\item if $c<\infty$, then $\Lambda_{u_0}=\infty$ and,   $$ \lim_{x\rightarrow \pm\infty} u(x)=\infty, \quad
 \lim_{x\rightarrow \pm\infty} u'(x) = \pm\sqrt{\mathrm{e}^{2(c-z_0)}-1}.$$
 \item if $c=\infty$, 
 $$ \lim_{x\rightarrow \pm\Lambda_{u_0}} u(x)=\infty, \quad
 \lim_{x\rightarrow \pm\Lambda_{u_0}} u'(x) = \pm\infty.$$
In particular, if $\Lambda_{u_0}<\infty$, the graph of $u$ is asymptotic to two vertical lines.
\end{itemize} 
\end{theorem}
\begin{theorem}\label{t3}
 Let $\varphi:\ ]a,\infty[\ \longrightarrow \ ]b,c[$,  $ a,b\in \{\R,-\infty\}$,  $c\in \{\R,\infty\}$ be a strictly decreasing diffeomorphism, then the solution $u$  of \eqref{htrans}-\eqref{ci} is defined in $]-\Lambda_{u_0},\Lambda_{u_0}[$, $\Lambda_{u_0} \in \{\R^+,\infty\}$, it is concave,  symmetric about the $y$-axis and has a maximum at $x=0$. Moreover, 
\begin{itemize}
\item if $c<\infty$, then 
$\Lambda_{u_0}<\infty $
  and,   $$ \lim_{x\rightarrow \pm\Lambda_{u_0}} u(x)=a, \quad
 \lim_{x\rightarrow \pm\Lambda_{u_0}} u'(x) = \pm\sqrt{\mathrm{e}^{2(c-z_0)}-1}.$$
\item if $c=\infty$, then 
$$\Lambda_{u_0}<\infty \iff \int_a^{u_0}\mathrm{e}^{-\varphi(\lambda)} d\lambda < \infty, $$
  and,   $$ \lim_{x\rightarrow \pm\Lambda_{u_0}} u(x)=a, \quad
 \lim_{x\rightarrow \pm\Lambda_{u_0}} u'(x) = \pm\infty.$$
\end{itemize}
\end{theorem}
\begin{remark}{\rm 
In the hypothesis  of Theorem \ref{t3}, the graph of $u$ is complete when $a=-\infty$. But in this case, by changing $\varphi$ by $-\varphi$, we can also apply Theorem \ref{t2}.} 
\end{remark}
\begin{definition} {\rm For each  solution $u$ of \eqref{htrans}-\eqref{ci}  we refer 
$ {\cal C}:= \text{Graph}(u) \times \R$
as a {\sl {\bf $[\varphi,\vec{e}_3]$-catenary cylinder}} surface.}
\end{definition}

\begin{figure}[h]
\begin{center}
\includegraphics[width=0.3\linewidth]{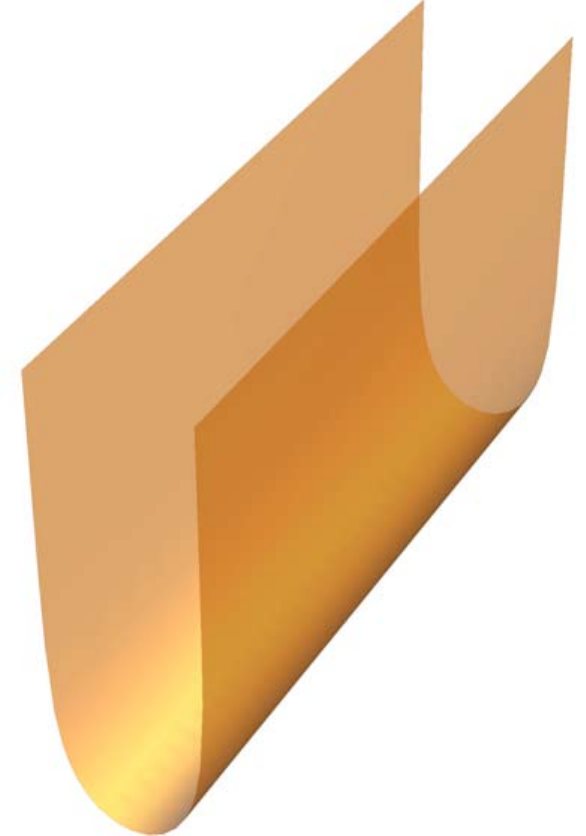}\qquad
\includegraphics[width=0.5\linewidth]{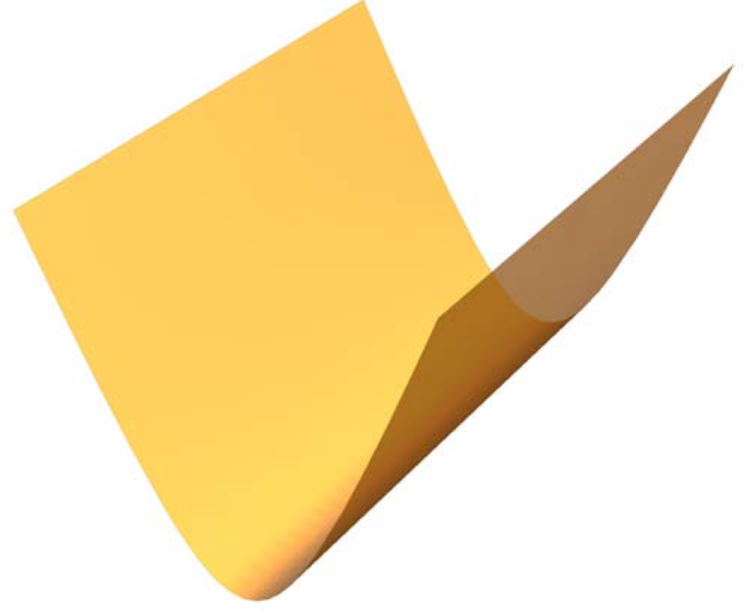}
\end{center}
\caption{$[\varphi,\vec{e}_{3}]$-catenary cylinders with $\dot{\varphi}=1$ and $\dot{\varphi}=1/u^2$, respectively.}
\end{figure}

\subsection{Tilted $[\varphi,\vec{e}_3]$-catenary cylinders}
 Let $\psi:=(x,y,u(x))$, $x\in]-\Lambda_{u_0},\Lambda_{u_0}[$  be a $[\varphi,\vec{e}_3]$-catenary cylinder with $u$ satisfying \eqref{ci} and Gauss map,
 $$ N = \frac{1}{\sqrt{1+u'^2}}(u',0,-1). $$
If we rotate the surface by an angle $\theta\in ]0,\pi/2[$ about the $x$-axis and dilate by $1/\cos\theta$, the resulting surface may be written as follows,
$$\widetilde{\psi}=\psi+\frac{1-\cos\theta}{\cos\theta}\langle \psi,\vec{e}_1\rangle \vec{e}_1+(\tan\theta)\vec{e}_1\wedge\psi,$$
where $\vec{e}_1=(1,0,0)$ and 
whose Gauss map  is given by,
\begin{equation}\widetilde{N}=\cos\theta \ N+(1-\cos\theta)\langle N,\vec{e}_1\rangle \vec{e}_1+\sin \theta \ \vec{e}_1\wedge N.\label{ntg}\end{equation}
The mean curvature $\widetilde{H}$ of $\widetilde{\psi}$  verifies
$$\widetilde{H}=\cos\theta \ H=-\cos \theta \ \dot{\varphi} \langle\vec{e}_3, N\rangle=-\dot{\varphi} \langle\vec{e}_3,\widetilde{N}\rangle.$$
Consequently, $\widetilde{\psi}$ is also $[\varphi,\vec{e}_3]$-minimal and we are going to refer these examples as  {\sl {\bf tilted $[\varphi,\vec{e}_3]$-catenary cylinders}}.
\\

Observe that,
\begin{equation}\widetilde{\psi}(x,y):=\left(\frac{x}{\cos \theta},y-u(x)\tan \theta,u(x)+y\tan \theta \right),\label{itg}\end{equation}
and it is the graph of the function 
\begin{align}
& {\cal C}_\theta: \  ]-\frac{\Lambda_{u_0}}{\cos \theta},\frac{\Lambda_{u_0}}{\cos \theta}[ \times \R \longrightarrow \R \nonumber\\
&{\cal C}_\theta (x,y) =\frac{u(x \cos \theta)}{\cos^2\theta} + y \tan\theta\nonumber
\end{align}

\begin{figure}[htb]
\begin{center}
\includegraphics[width=0.3\linewidth]{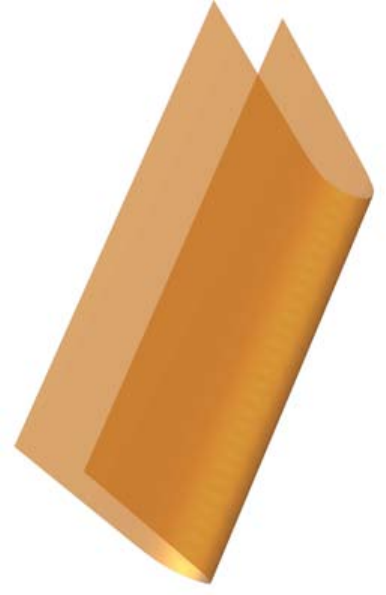}\qquad
\includegraphics[width=0.32\linewidth]{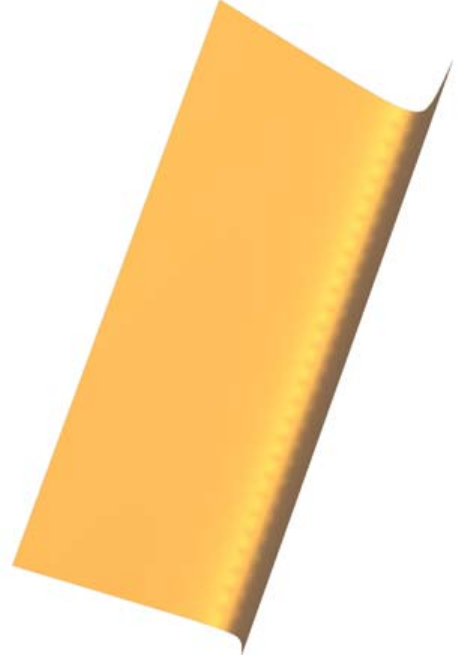}
\end{center}
\caption{Titled $[\varphi,\vec{e}_{3}]$-catenary cylinders with $\dot{\varphi}=1$ and $\dot{\varphi}=1/u^3$, respectively.}
\end{figure}

\begin{theorem}
Let $\Sigma\subset \R^3$ be  a complete flat  $[\varphi,\vec{e}_3]$-minimal surface. If $\varphi:\R\rightarrow \R$ is a strictly increasing diffeomorphism, then $\Sigma$ is either a vertical plane or a $[\varphi,\vec{e}_3]$-catenary cylinder (maybe tilted) surface. 
\end{theorem}
\begin{proof}
From basic differential geometry, $\Sigma=\alpha \times \Pi^\perp$ is a ruled surface and its Gauss map is constant along the rules, where $\alpha$ is a complete regular curve in a plane $\Pi\subset \R^3$.
\begin{quote}{\sc Claim:}
Let ${\cal L}$ be a straight line of  $\Sigma$  and  ${\cal V}_L$ be the unit normal vector along ${\cal L}$. If $\langle{\cal V}_L,\vec{e}_3\rangle\neq 0,$ then there exists a $[\varphi,\vec{e}_3]$-catenary cylinder $\cal C_{\cal L}$ (tilted, if ${\cal L}$ is not horizontal) containing ${\cal L}$ and tangent to $\Sigma$ along ${\cal L}$.
\end{quote}
Then, up to  an appropriate  rotation and dilatation, $\Sigma$ is tangent to a $[\varphi,\vec{e}_3]$-catenary cylinder along a rule.  The result follows from standard theory of uniqueness of solution for the ODE \eqref{htrans}.
\end{proof}
\begin{proof}[Proof of the claim]
If $\CL$ is horizontal then, after a rotation about the axis $\vec{e}_3$, we may assume that 
$$ \CL = \{(x_0,0,u_0) + s (0,1,0) \ | \ s\in \R\}$$
and there exists $\phi\in]-\pi/2,\pi/2[$ such that $\CV_{\CL}= (-\sin \phi, 0, \cos \phi)$. Then, as $\varphi:\R\rightarrow\R$ is a strictly increasing diffeomorphism, from \eqref{diagramafase}, there exists a solution $u_{\CL}$ of  \eqref{htrans}-\eqref{ci} and $x_1\in \R$, such that $u_{\CL}(x_1)=u_0$ and $u_{\CL}'(x_1)=\tan \phi$. The $[\varphi,\vec{e}_3]$-catenary cylinder we are looking for is just a translation in the $\vec{e}_1$-axis of the catenary cylinder ${\cal C}_{u_{\CL}}$ associated to $u_{\CL}$.
\\

If $\CL$  is not horizontal and $p=\CL\cap \{z=0\}$, then by rotation of center $p$ and axis $\vec{e}_3$ we may assume there exists $\theta\in]-\pi/2,0[$ and $\alpha \in \R$, such that 
$$\CV_{\CL} = \frac{1}{\sqrt{\alpha^2 + 1}}(-\alpha,-\sin \theta, \cos \theta).$$
So, from \eqref{ntg} and \eqref{itg}, if we take the solution $u_{\CL}$ of  \eqref{htrans}-\eqref{ci} satisfying 
$$ u_{\CL}(x_1) = \langle p,\vec{e}_2\rangle \ \cos \theta \ \sin \theta, \qquad  u_{\CL}'(x_1)=\alpha,$$
for some $x_1\in \R$, we conclude that our tilted $[\varphi,\vec{e}_3]$-catenary cylinder is a translation in the $\vec{e}_1$-axis of the tilted $[\varphi,\vec{e}_3]$-catenary cylinder obtained after rotation of angle $\theta$ around the $\vec{e}_2$-axis and dilation of $1/\cos \theta$ the $[\varphi,\vec{e}_3]$-catenary cylinder associated to $u_{\CL}$.
\end{proof}

As consequence, from the Theorem \ref{pmcurvaturagauss}, the following result holds,

\begin{corollary}
Let $\varphi:]a,b[\rightarrow\mathbb{R}$ be  a strictly increasing  function satisfying  $\dddot{\varphi}\leq0$, 
and let $\Sigma$  be a complete locally convex $[\varphi,\vec{e}_3]$-minimal immersion in $\mathbb{R}^{2}\times ]a,b[$. If the Gauss curvature $K$ vanishes anywhere, then $\Sigma$ is either a vertical plane or a $[\varphi,\vec{e}_{3}]$-catenary cylinder (maybe tilted) surface.
\end{corollary}

\section{ $[\varphi,\vec{e}_3]$-minimal surfaces of revolution}
In this section we are going to study geometric behaviour of rotationally symmetric solutions of \eqref{fe}.   
\subsection{The singular case}
In the rotationally symmetric case, the equation \eqref{fe} reduces to the following ordinary differential equation for $u=u(r)$, $r=\sqrt{x^2+y^2}$:
\begin{equation}
\label{equationrot}
u^\pp=\left(1+u^{\p 2} \right)\left(\dot{\varphi}(u)-\frac{u^\p}{r} \right), 
\end{equation}
where $(^\p)$ denotes derivative respect to $r$ and $\varphi:]a,b[\subseteq \R\longrightarrow \R$ is a smooth function.  Since \eqref{equationrot} is degenerated,  the existence and uniqueness of solution at $r=0$ is not assured by standard theory. Multiplying by $r$ we obtain that  \eqref{equationrot} also writes as,
\begin{equation}
\label{eqrot}
\left(\frac{r \ u^\p}{\sqrt{1+u^{\p 2}}} \right)^\p = \frac{r \dot{\varphi}(u)}{\sqrt{1+u^{\p2}}}.
\end{equation}
But, from \cite[Theorem 2]{Serrin},  a solution of \eqref{fe}  cannot possess isolated non-removable singularities, hence, it is not a restriction  to look for the existence of solutions of \eqref{eqrot} with the following initial conditions: 
\begin{equation}\label{ic}
 u(0) = u_0\in]a,b[, \qquad u^\p(0) =0.
 \end{equation}
In this sense and  by using a similar argument to  \cite[Proposition 2]{Rafa} we can assert 
 \begin{proposition}\label{epro}
The problem \eqref{equationrot}-\eqref{ic} has a unique solution $u\in {\cal C}^2([0,R])$ for some $R>0$ which depends continuously on the initial data and such that 
 $$ u^\pp(0) = \frac{\dot{\varphi}(u_0)}{2}.$$
 \end{proposition}
 The following result allows us to compare rotational symmetric $[\varphi,\vec{e}_3]$-minimal graphs,
\begin{proposition}
\label{controlinf}
Let $\varphi_1, \varphi_2:]a,b[\rightarrow \R$   be strictly increasing and convex functions satisfying that $\dot{\varphi}_1>\dot{\varphi}_2$ on $]a,b[$ and denote by  $u_{\varphi_1}$  and   $u_{\varphi_2}$  the  $[\varphi_i,\vec{e}_3]$-minimal graphs solutions to the corresponding problem \eqref{equationrot}-\eqref{ic}. Then   $$u_{\varphi_1}^\p> u_{\varphi_2}^\p, \qquad  \text{on $]0,r_0[$}.$$
\end{proposition}
\begin{proof}
If we take the function $d:=u_{\varphi_1}^\p-u_{\varphi_2}^\p$, then   $d(0)=0$ and $$d\,^\p(0)=u_{\varphi_1}^\pp(0)-u_{\varphi_2}^\pp(0)=\left(\frac{\dot{\varphi}_{1}(u_{0})}{2}-\frac{\dot{\varphi}_{2}(u_{0})}{2}\right)>0.$$
Hence, there exists $\epsilon>0$ such that $d = u_{\varphi_1}^\p-u_{\varphi_2}^{\p}>0$  on  $]0,\epsilon[$. If there exists $r_{1}>0$ satisfying  $d(r_1)\leq 0$, we can take  $r^{*}:=\inf\{r>0: d(r)<0\}$ so that   $d(r^{*})=0$ and $d^{\,\p}(r^{*})\leq 0$. But, from \eqref{equationrot} and  having in mind that  $\int_0^{r^*}d > 0$, we get
\begin{align*}0\geq d\,^\p(r^{*})& =(1+u_{\varphi_1}^\p(r^{*})^2)\left[ \dot{\varphi}_{1}(u_{\varphi_1}(r^*))-\dot{\varphi}_{2}(u_{\varphi_2}(r^*)\right]\\
& >(1+u_{\varphi_1}^\p(r^{*})^2)\left[ \dot{\varphi}_{1}(u_{\varphi_2}(r^*))-\dot{\varphi}_{2}(u_{\varphi_2}(r^*)\right] > 0,
\end{align*}
which is a contradiction.
\end{proof}
\begin{remark}{\rm 
The above Proposition also holds if we assume that $\varphi_1, \varphi_2:]a,b[\rightarrow \R$  are smooth functions so that 
$$ \inf{\dot{\varphi}_{1}}>\sup{\dot{\varphi}_{2}}, \qquad \text{on $]a,b[$}.$$}
\end{remark}
As consequence of Proposition \ref{controlinf} and the asymptotic behavior of rotational solitons proved in  \cite{CSS} we have
\begin{corollary}
Let $\varphi:[a,+\infty[\rightarrow \R$ be strictly increasing regular function and $u$ be an entire solution of \eqref{equationrot}. If  there exists $\alpha>0$ such that $\dot{\varphi}>\alpha$, then 
$$u^\p(r)\geq  \alpha\, r-\frac{1}{\alpha\,r}, $$
for $r$ large enough.
\end{corollary}
\subsection{Geometric description of revolution $[\varphi,\vec{e}_3]$-minimal surfaces}
Now, we want to describe  $[\varphi,\vec{e}_3]$-minimal surfaces that  are invariant under  the one-parameter group of rotations that fix the $\vec{e}_{3}$ direction. A such surface with generating curve the arc-lenght parametrized curve $$\gamma(s)=(x(s), 0, z(s)), \qquad   \ \ s\in I\subset\mathbb{R}$$ is given by,
\begin{equation}
\label{param}
\psi(s,t)=\left(x(s)\cos(t), x(s)\sin(t),z(s) \right), \emph{ } (s,t)\in I\times\mathbb{R}.
\end{equation}

The inner normal of $\psi$ writes as
\begin{equation}
\label{Gaussmap}
N(s,t)=\left(-z'(s)\cos(t),-z'(s)\sin(t) ,x'(s) \right),
\end{equation}
and the coefficients of the first and second fundamental form,
\begin{equation}
\label{coefi}
\begin{array}{lll}
&\langle\psi_{s},\psi_{s}\rangle=1,  & \langle\psi_{s},N_{s}\rangle =-\kappa, \\
&\langle\psi_{t},\psi_{t}\rangle=x^{2}, &\langle\psi_{t},N_{t}\rangle = -x \, z', \\ 
&\langle\psi_{s},\psi_{t}\rangle=0, &\langle\psi_{s},N_{t}\rangle =0,
\end{array}
\end{equation}
where $\kappa$ is the curvature of ${\gamma}$ and by $'$ we denote derivative respect to $s$.
\\

From \eqref{coefi}, the mean curvature vector of $\psi$ is given by
\begin{equation}
\label{Hvector}
{\text {\bf H}}=-\left(\kappa+\frac{z'}{x} \right)N.
\end{equation}
Consequently, from \eqref{me}, \eqref{param}  and \eqref{Gaussmap}, the surface $\psi$ is a $[\varphi,\vec{e}_{3}]$-minimal surface if and only if
\begin{equation}
\label{equz}
\left\{ \begin{array}{l}
x'=\cos(\theta)\\
z'=\sin(\theta),\\
\theta'=\dot{\varphi}(z)\text{cos}(\theta)-\frac{\text{sin}(\theta)}{x},
\end{array}\right.
\end{equation}
where  $\theta(s)=\int_0^s \kappa(t) dt$.

\

Along this section we will consider that 
  $\varphi:\ ]a,\infty[\ \longrightarrow \R$ is a strictly increasing and convex function, that is
  \begin{equation}
  \label{conditions}
  \dot{\varphi}>0, \quad \ddot{\varphi}\geq0,  \qquad \text{on $]a,\infty[$}.
  \end{equation}
\subsubsection{Globally convex examples}
Here, we want to study the solutions of \eqref{equz} with the following initial conditions,
\begin{equation}
\label{valorinicial} x(0)=0, \qquad z(0)=z_{0}\in ]a,\infty[,  \qquad \theta(0)=0.
\end{equation}
In this case, the surface intersects orthogonally the rotation axis and we have the following result:
\begin{theorem}
\label{bowl} If $x_0=0$,  then ${\gamma}$ is the graph of a strictly convex  symmetric function $u(x)$ defined  on a maximal interval $]-\omega_+,\omega_+[$ which has a minimum at $0$ and $$\lim_{x\rightarrow\pm \omega_+}u(x)=\infty.$$ 
\end{theorem}
\begin{proof}
First of all, we remark that the existence of ${\gamma}$ around $s=0$ is guaranteed from Proposition \ref{epro}.

Moreover, it is easy to see that  $\overline{x}(s)=-x(-s)$, $\overline{z}(s)=z(-s)$ and $\overline{\theta}(s)=-\theta(-s)$ are also solutions of the same initial value problem \eqref{equz}-\eqref{valorinicial}. Hence, $\gamma$ is symmetric respect to $\vec{e}_{3}$ direction and we may consider only the case $s\geq 0$.

\

By application of  L'H\^{o}pital's rule, we have that $2 \theta'(0)=\dot\varphi(z_0)>0$ and $\gamma$ is a strictly locally convex planar curve around of $s=0$. We assert that   $\theta'(s)>0$ for $s\geq 0$, otherwise from \eqref{valorinicial}, there exists a first value $s_{0}>0$ such that $\theta'(s_{0})=0$ and $\theta''(s_{0})\leq 0$. As $\theta' >0$ on $[0,s_0[$, from \eqref{equz} we have that  $0<2 \theta(s_{0})<\pi$ and by  differentiation  of \eqref{equz}, we get,
\begin{equation}
\label{phisegundapositiva}
\theta''(s_{0})=\frac{\sin(2\theta(s_{0}))}{2}\left(\ddot{\varphi}(z(s_{0}))+\frac{1}{x(s_{0})^2} \right) > 0,
\end{equation}
getting a \textit{contradiction}. 

In the same way, as $\theta' >0$ for $s>0$,  we have that $0<2 \theta(s)<\pi$ for $s>0$ and $\gamma$ is the graph of a strictly convex function $u=u(x)$ which is a ${\cal C}^2$  solution of 
\begin{equation}
\left\{ \begin{array}{l}\label{des1}
u^\pp=\left(1+u^{\p 2} \right)\left(\dot{\varphi}(u)-\frac{u^\p}{x}\right) >0,\\
u(0) = z_0, \qquad u^\p(0) = 0,
\end{array}
 \right.
 \end{equation}
on the maximal interval of existence $]-\omega_+, \omega_+[$. Finally, if  $\text{lim}_{x\rightarrow\pm \omega_+}u(x)=h_{0}<\infty$, then the  standard theory of prolongation of solutions, gives that  $\omega_+=+\infty$ which is also a contradiction by the convexity of $u$. 
\end{proof}

\begin{figure}[h]\label{figurebowls}
\begin{center}
\includegraphics[width=0.32\linewidth]{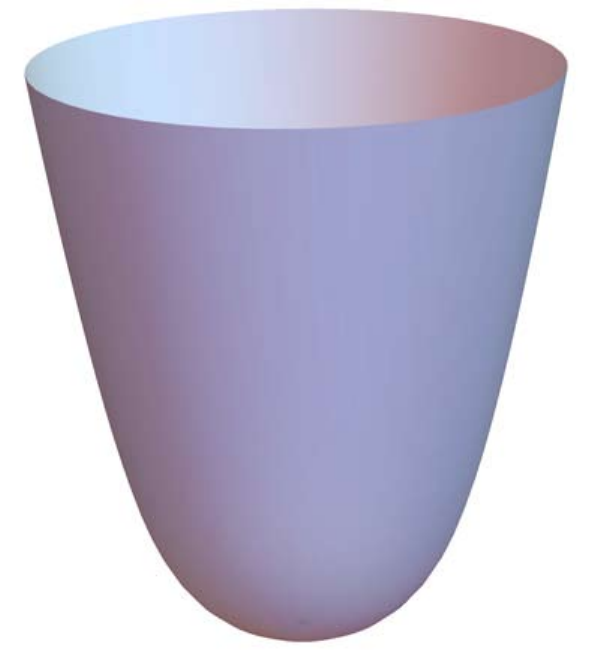}
\qquad \qquad 
\includegraphics[width=0.17\linewidth]{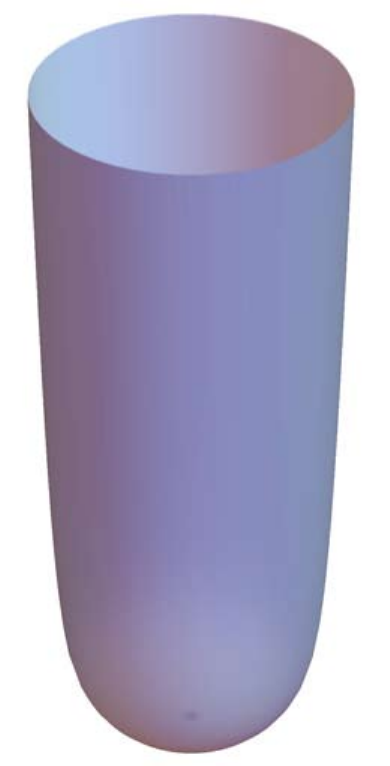}
\end{center}
\caption{$[\varphi,\vec{e}_{3}]$-minimal bowls for $\dot{\varphi}(u)=e^{-1/u}$(left) and $\dot{\varphi}(u)=u^2\text{ (right)}$.}
\end{figure}

\begin{definition} {\rm If $\gamma$ is a graph as in Theorem \ref{bowl}, we are going to say that the revolution surface with generating curve $\gamma$ is a} $[\varphi,\vec{e}_{3}]$-minimal bowl.
\end{definition}

\subsubsection{Non-convex examples}\label{ctypeexamples}
Now, we want to study the solutions of \eqref{equz} with the following initial conditions,
\begin{equation}
\label{condicionescatenoid} x(0)=x_0>0, \qquad z(0)=z_{0}\in ]a,\infty[,  \qquad \theta(0)=0.
\end{equation}

From standard theory, the existence and uniqueness of  solution to the  problem \eqref{equz}-\eqref{condicionescatenoid}  is guaranteed. 

\

Let $]-s_-,s_+[$ be the maximal interval of existence and 
consider $\gamma^{+}:=\gamma\big\vert_{[0,s_+[}$ the right branch of $\gamma$. Arguing as in Theorem \ref{bowl}, we can prove that $\gamma^+$ is the graph of a convex function $u=u(x)$ defined on a maximal interval $]x_0,\omega_+[$, such that $$\lim_{x\rightarrow  \omega_+}u(x)=\infty.$$

 For studying the left branch of $\gamma$ we are going to consider, $\gamma^-(s) = \gamma(-s)$ for $s\in [0,s_-[$. Then, by taking $\overline{x}(s)=x(-s)$,  $\overline{z}(s)=z(-s)$ and  $\overline{\theta}(s)=\theta(-s)+\pi$ for $s\in [0,s_-[$, we have that $\{ \overline{x}, \overline{z},\overline{\theta}\}$ is a solution of \eqref{equz} on $[0,s_-[$ satisfying
\begin{equation}
\label{condicionescatenoid2} \overline{x}(0)=x_0>0, \qquad \overline{z}(0)=z_{0}\in ]a,\infty[,  \qquad \overline{\theta}(0)=\pi.
\end{equation}
\begin{lemma}\label{lwl1}
There exists $s_0\in ]0,s_-[$ such that $2 \overline{\theta}(s_0) = \pi$.
\end{lemma}
\begin{proof}
Assume on the contrary,  $ \overline{\theta}(s)\in ] \frac{\pi}{2},\pi[$ for all  $s\in]0,s_-[$ and,  from \eqref{equz}-\eqref{condicionescatenoid2}, we have that $\overline{x}\,'<0$,  $\overline{\theta}\,'<0$   and $\overline{z}\,'>0$ on $]0,s_-[$. Hence, there exist  
$$\overline{x}_- = \lim_{s\rightarrow s_-} \overline{x}(s), \quad \overline{z}_- = \lim_{s\rightarrow s_-} \overline{z}(s), \quad
\overline{\theta}_- = \lim_{s\rightarrow s_-} \overline{\theta}(s),$$
and, as $]-s_-,s_+[$ is the maximal interval of existence of $\gamma$, we have that either $\overline{x}_-=0$ or $\overline{z}_-=\infty$.  
So,  $\gamma^-$ is the graph of a convex function $\overline{u}=\overline{u}(\overline{x})$ on $]\overline{x}_-,x_0[$ such that either $\overline{x}_-=0$ or $\lim_{\overline{x}\rightarrow\overline{x}_-}\overline{u}(\overline{x}) = +\infty$.

In the first case, if $\lim_{\overline{x}\rightarrow\overline{x}_-}\overline{u}(\overline{x}) =+ \infty$, from the convexity of $\overline{u}$ we get that $\overline{\theta}_-=\frac{\pi}{2}$ and there exists a sequence $\{s_n\}\rightarrow s_-$ satisfying $\overline{\theta}\,'(s_n) \rightarrow 0$, but then, from \eqref{equz},
\begin{align*} 0 &= \lim_{n\rightarrow\infty}\overline{\theta}\,'(s_n) = \cos(\overline{\theta}(s_n)) \dot{\varphi}(\overline{z}(s_n)) - \frac{\sin(\overline{\theta}(s_n))}{\overline{x}(s_n)} \\
& \leq  \cos(\overline{\theta}(s_n)) \dot{\varphi}(\overline{z}(s_n))  \leq 0.\end{align*}
Thus, 
$$ 0 = \lim_{n\rightarrow\infty} \cos(\overline{\theta}(s_n)) \dot{\varphi}(\overline{z}(s_n)) = \lim_{n\rightarrow\infty}\frac{\sin(\overline{\theta}(s_n))}{\overline{x}(s_n)} = \frac{1}{\overline{x}_-}\neq 0,$$
which is a contradiction.

\

If $\overline{x}_-=0$ then, from \cite[Theorem 2]{Serrin}, $\lim_{\overline{x} \rightarrow 0}\overline{u}(\overline{x}) = +\infty$ and arguing as above we also obtain a contradiction.
\end{proof}
\begin{lemma}\label{lwl2} If 
   $s\in]s_0,s_-[$, then $0<2\overline{\theta}(s)<\pi$.
\end{lemma}
\begin{proof}
It is clear because $\overline{\theta}\,' <0$ on $\overline{\theta}\,^{-1}(\frac{\pi}{2})$ and  $\overline{\theta}\,' >0$ on $\overline{\theta}\,^{-1}(0)$.
\end{proof}
\begin{lemma}\label{lwl3}  $\overline{\theta}$  has a  minimum at a point $s_1\in ]s_0,s_-[$ and   $\overline{\theta}\,' >0$ on $]s_1,s_-[$ 
\end{lemma}
\begin{proof}
Assume that $\overline{\theta}\,' <0$ on $]s_0,s_-[$. Then, from Lemma \ref{lwl2}, $\overline{x}\nearrow \overline{x}_-$,  $\overline{z}\nearrow \overline{z}_-$ and $\overline{\theta}\searrow\overline{\theta}_-\in[0,\frac{\pi}{2}[$ when $s\rightarrow s_-$.  In particular, there is a sequence $\{s_n\}\rightarrow s_-$ satisfying $ \lim_{n\rightarrow \infty} \overline{\theta}\,' (s_n) =0$.

\

Under this assumption,  we assert that  $\overline{\theta}_-\neq 0$ and $\overline{x}_-<+\infty$, otherwise 
\begin{align*} 0 &= \lim_{n\rightarrow \infty} \overline{\theta}\,' (s_n) = \cos(\overline{\theta}(s_n)) \dot{\varphi}(\overline{z}(s_n)) - \frac{\sin(\overline{\theta}(s_n))}{\overline{x}(s_n)} \\
& =\cos(\overline{\theta}_-) \lim_{n\rightarrow \infty}\dot{\varphi}(\overline{z}(s_n))  \geq \cos(\overline{\theta}_-) \dot{\varphi}(z_0) > 0,
\end{align*}
which is a contradiction. Thus,   $\gamma^-$ is the graph of a concave function $\overline{u}=\overline{u}(\overline{x})$ on a bounded interval $]\overline{x}(s_0),\overline{x}_-[$ satisfying  $\lim_{\overline{x}\rightarrow\overline{x}_-}\overline{u}(\overline{x}) = +\infty$ but this is also a contradiction because $\overline{\theta}$ is strictly decreasing on $]s_0,s_-[$.

\

Hence there exists $s_1\in ]s_0,s_-[$ such that $\overline{\theta}\,' (s_1)=0$. Moreover, from \eqref{equz},
$$ \overline{\theta}\,'' (s_1) = \sin(\overline{\theta}(s_1)) \cos(\overline{\theta}(s_1))(\ddot{\varphi}(\overline{z}(s_1)) + \frac{1}{\overline{x}(s_1)})>0,$$
and $s_1$ is a local minimum of $\overline{\theta}$.  Now, arguing as in Theorem \ref{bowl}, we can prove that,  on  the interval $]\overline{x}(s_1),\overline{x}_-[$,  $\gamma^-$ is the graph of a convex function satisfying $$\lim_{\overline{x}\rightarrow  \overline{x}_ -}\overline{u}(\overline{x})=+\infty.$$
\end{proof}

\begin{lemma}
\label{lwl4}
The profile curve $\gamma$ is embedded.
\end{lemma}
\begin{proof}
Let $s_{0}\in ]-s_{-},0[$ the point given by the Lemma \ref{lwl1} and consider the following branches of $\gamma$ determinated by $\gamma\big\vert_{-s_{-},s_{0}[}$ and $\gamma\big\vert_{ ]s_{0},s_{+}[}$ respectively, parametrized by
\begin{align*}
&u_{+}(x)=(x,u_{+}(x)) \text{ for any } x\in ]x_{0},x(s_{-})[  \\
& u_{-}(s)=(x,u_{-}(x)) \text{ for any } x\in ]x_{0},x(s_{+})[,
\end{align*}
where $u$ is solution of the equation \eqref{des1}. Now, define the following smooth function $d(x)=u^{\p}_{+}(x)-u^{\p}_{-}(x)$. It is clear that $d(x)>0$ for $x\in ]x_{0},x_{0}+\delta[$ for some $\delta>0$. Suppose that there exists a first $r\geq x_{0}+\delta$ such that $d(r)=0$ and $d^{\p}(r)\leq 0$. Consequently, $u_{+}(x)>u_{-}(x)$ for any $x\in ]x_{0},r[$  and from the equation \ref{des1}, we get to contradiction since,
$$d^{\p}(r)=(1+u^{\p}(r)^{2})\left(\dot{\varphi}(u_{+}(r))-\dot{\varphi}(u_{-}(r)) \right)>0.$$
Thus, $d^{\p}>0$ everywhere and integrating  $u_{+}(x)>u_{-}(x)$ for any $x\geq x_{0}$.
\end{proof}

\begin{figure}[h]
\begin{center}
\includegraphics[width=0.35\linewidth]{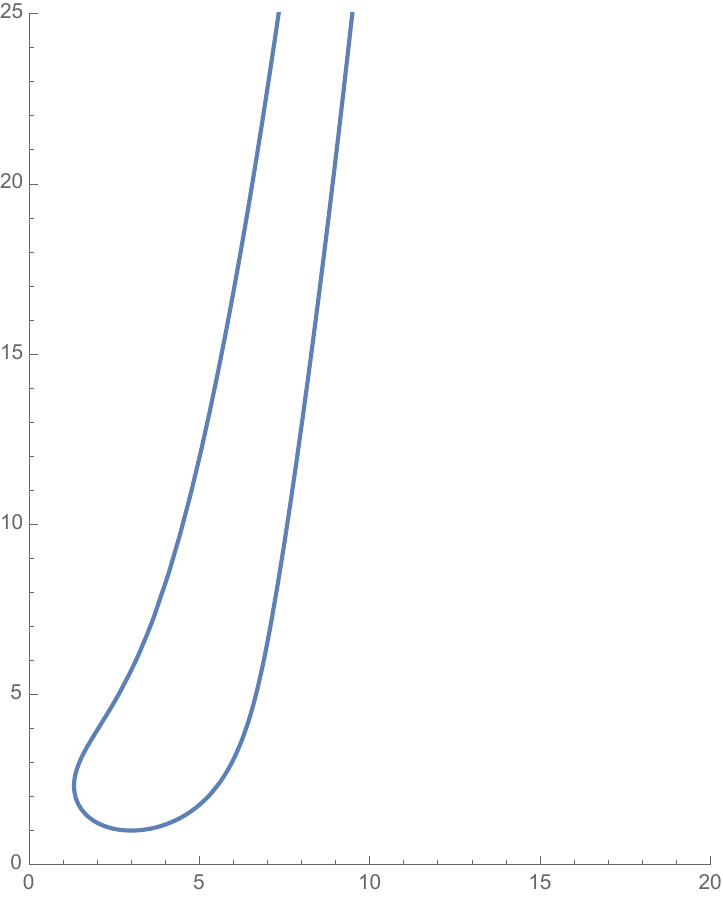}
\quad  
\includegraphics[width=0.38\linewidth]{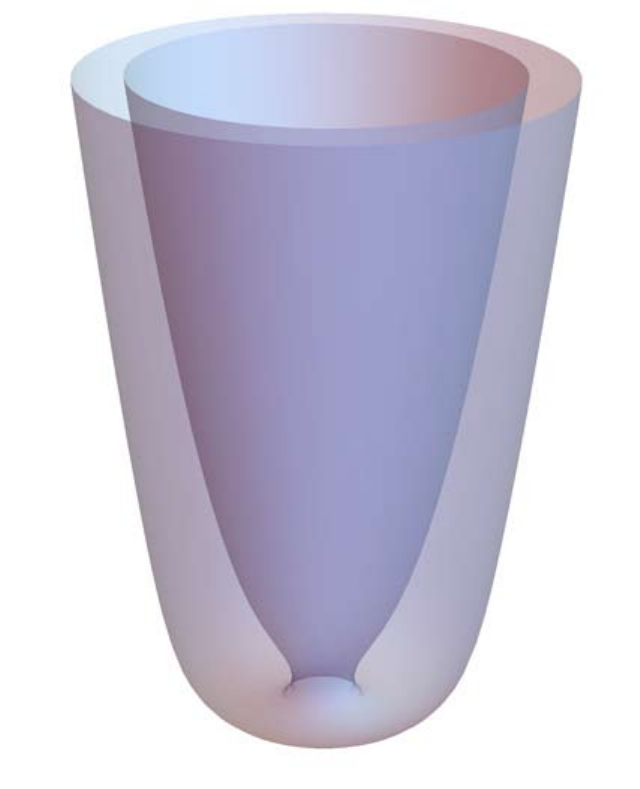}
\end{center}
\caption{$[\varphi,\vec{e}_{3}]$-minimal catenoid with $\dot{\varphi}(u)=e^{-1/u}$.}\label{fcatenoid}
\end{figure}

\begin{theorem}
\label{existencecatenoid}
For every  $x_{0}>0$, there exists a complete embedded rotational $[\varphi,\vec{e}_{3}]$-minimal, see Figure \ref{fcatenoid} (right) with the annulus topology whose distance to axis of revolution is $x_{0}$ and whose generating curve $\gamma$ is of winglike type see Figure \ref{fcatenoid} (left).

These examples will be called  $[\varphi,\vec{e}_{3}]$-minimal {\sl catenoids}.
\end{theorem}
\begin{proof}
It follows from Lemma \ref{lwl1},  Lemma \ref{lwl2}, Lemma \ref{lwl3} and Lemma \ref{lwl4}.
\end{proof}

\begin{proposition}
\label{intervalodef}
Under the above conditions, the following staments hold:
\begin{enumerate}
\item If $\dot{\varphi}$ has at most a linear growth, then $\omega_{+}=+\infty$ and $\overline{x}_-=+\infty$.
\item If $\dot{\varphi}$  growths as  $u^{\alpha}$ for some  $\alpha>1$, then $\omega_+ , \overline{x}_-\in \R$.
\end{enumerate}
\end{proposition}

\begin{proof}

\

\noindent If $\dot{\varphi}$ has at most a linear growth, then there must be a constant $c>0$ such that $\dot{\varphi}(u)/u\leq c$ outside a compact set. Thus,  from the inequality \eqref{des1}, when  $x$ is large enough the following inequalities hold,
\begin{equation}
\label{ineq}
x\geq \frac{u^\p}{\dot{\varphi}(u)}(x)\geq \frac{1}{c} \frac{u^\p}{u}(x).
\end{equation}
Integrating both members of the inequality \eqref{ineq}, we get that,
\begin{equation}
\frac{x^{2}}{2}-\frac{x_{0}^{2}}{2}\geq\frac{1}{c}\, \text{log}\left(\frac{u(x)}{u(x_{0})}\right) \emph{ for some } x_{0}>0.
\end{equation} 
Hence, $\omega_{+}
=+\infty$ and $\overline{x}_-=+\infty$.

\

\noindent Let's go to consider now  that 
$$\lim_{u\rightarrow +\infty}\frac{\dot{\varphi}(u)}{u^{\alpha}} = M\neq 0\quad \text{for some $\alpha>1$},$$ and suppose  that $\omega_{+}=+\infty$. Then, from the Theorem \ref{bowl} and Theorem \ref{existencecatenoid}, the real function $f$ given by,
$$
f(r):=\frac{u'(r)}{M\, u^\alpha(r)}
$$
has, for $r$ large enough, a bounded and  strictly monotone primitive $F(u)(r)$.   
Hence,  there exists a sequence $\{ r_{n} \}\nearrow +\infty$ such that 
\begin{equation}\label{lim0}
\lim_{n\rightarrow \infty} f(r_{n}) =  0.
\end{equation}
\begin{afirmacion}\label{cl0} The function $f$ satisfies that 
$\lim_{r\rightarrow \infty} \frac{f(r)}{r} =  0$.
\end{afirmacion}
\begin{proof}[Proof of  Claim \ref{cl0}]
Assuming on the contrary, there exists $\delta>0$ and a sequence $\{s_n\}\nearrow+\infty$ such that 
$$f(s_n) > \frac{f(s_n)}{s_n} > \delta,$$
which together \eqref{lim0}, says that  $f^{-1}(\delta)$ is unbounded real subset containing  a  divergent sequence to $+\infty$.

But, from the equation \eqref{equationrot}, the function $f$ satisfies the following differential equation
\begin{equation}
\label{equdiff}
f^{\,\p} =\left(\frac{\dot{\varphi}(u)}{M\, u^{\alpha}}-\frac{f(r)}{r} \right)+M^2\, f^{2}\, u^{2 \alpha}\, \left(\frac{\dot{\varphi}(u)}{M\, u^{\alpha}}-\frac{f}{r} -\frac{\alpha}{M}u^{-\alpha+1} \right)
\end{equation}
and we obtain that there exists $\hat{r}\in f^{-1}(\delta)$ such that $f^{\,\p}(r)>1$  for any $r\in f^{-1}(\delta)$, $r\geq\hat{r}$,  which is impossible because $f^{-1}(\delta)$ is unbounded.
\end{proof}

From \eqref{equdiff}, Claim \ref{cl0} and using that $u$ diverges to $+\infty$ we get that, for $r$ sufficiently large, the following inequality holds ,
\begin{equation}
\label{inecon}
\frac{2f^{\,\p}}{1+f^{2}} >1.
\end{equation}
By  integration of this expression,
we conclude  that $\omega_{+}<+\infty$.
\end{proof}
\begin{remark}{\rm 
Notice that $\omega_{+}=+\infty$ does not imply that $\dot{\varphi}$ has at most a linear growth. For example, by taking $\dot{\varphi}(u)=u\,\text{log}(u)$ with $u\geq 1$ and by  the integration of both members in  \eqref{des1}, we get that,
$$\frac{x^{2}}{2}-\frac{x_{0}^{2}}{2}\geq \text{log}\left(\text{log} \left(\frac{u(x)}{u(x_{0})}\right)\right) \emph{ for some } x_{0}>0.$$
Thus, $\omega_{+}=+\infty$ but the function $\log(u)$ is not bounded.}
\end{remark}

\section{Asymptotic behavior of rotational examples}  
Clutterbuck, Schn\"urer and Schulze studied  in \cite{CSS} the asymptotic behavior of solitons rotationally symmetric. They proved that the problem 
\begin{equation}
\label{ecuacion}
\left\lbrace
\begin{array}{ll}
u^\pp=(1+u^{\p 2})\left(1-\frac{u^\p}{r}\right), &r > R,\\
u(R)=u_{0}\in \R, & u^\p(R)=u_1\in \R. 
\end{array}
\right.
\end{equation}
has a unique ${\cal C}^\infty$-solution $u$ on $[R,\infty[$. Moreover, as $r \rightarrow \infty$, $u$ has the following  asymptotic
expansion
$$u(r)=\frac{r^{2}}{2}-\text{log}(r)+ {O}(r^{-2}).$$

Due to the arbitrariness of the problem \eqref{equationrot}  it is impossible  to find  a general  asymptotic behavior of their solutions because if you consider any strictly convex smooth function $u=u(r)$, $r>R$, one can find a function $\varphi$ such that $u$ is a solution of \eqref{equationrot}.

Proposition \ref{intervalodef} motivates to consider  $\varphi:]a,+\infty[\longrightarrow \R$ a smooth  function satisfying \eqref{conditions} and with  a quadratic growth, that is,  with the following  asymptotic behavior,
\begin{align}\label{clinear}
\lim_{u\rightarrow \infty} \ddot{\varphi}(u) = \alpha \geq 0 \quad \text{and} \quad \lim_{u\rightarrow \infty} (\dot{\varphi}(u) - \alpha\, u) = \beta\in\mathbb{R}.
\end{align}
In this case,  we  are going to generalize the  result in \cite{CSS} to the following problem,
\begin{equation}
\label{ab}
\left\lbrace
\begin{array}{ll}
u^\pp=(1+u^{\p 2})\left( \dot{\varphi}(u)-\frac{u^\p}{r}\right), &r > r_0\geq 0,\\
u(r_0)=u_{0}> a, & u^\p(r_0)=u_1\geq 0, 
\end{array}
\right.
\end{equation}
with   ${\varphi}:]a,\infty[\longrightarrow \R$ satisfying  \eqref{conditions} and \eqref{clinear}.
\begin{remark}\label{betapositive}
Observe that if $\alpha>0$, then $u$ is solution of \eqref{equationrot}
if and only if $ v= u + \frac{\beta-\widetilde{\beta}}{\alpha}$ is  solution of 
$$ v^\pp=(1+v^{\p 2})\left( \dot{\psi}(v)-\frac{v^\p}{r}\right)$$
where 
$ {\psi}(v) = {\varphi}\left(v-\frac{\beta-\widetilde{\beta}}{\alpha}\right)$ satisfies
\begin{align*}
\lim_{v\rightarrow \infty} \ddot{\psi}(v) = \alpha \geq 0 \quad \text{and} \quad \lim_{v\rightarrow \infty} (\dot{\psi}(v) - \alpha\, v) = \widetilde{\beta}.
\end{align*}
It is also clear that $\frac{v^\p}{\dot{\psi}(v)} = \frac{u^\p}{\dot{\varphi}(u)}$.
\end{remark}
\begin{theorem}[{\bf Case $\alpha>0$}]
\label{comportamientoasin}
Assume  that $\dot{\varphi}(u_0)\, r_0 \geq u_1$ and $\alpha>0$.  Then the problem \eqref{ab} has an unique strictly convex ${\cal C}^\infty$-solution $u$ on $[r_0,\infty[$. Moreover, as $r \rightarrow \infty$, we have the following asymptotic expansion:
\begin{align}
&\dot{\varphi}(u)(r)=e^{\frac{1}{2}\alpha r^{2}+o\left(r^2\right) }\\
&\frac{u^\p}{\dot{\varphi}(u)} (r)=r- \alpha\, r\, \dot{\varphi}(u)^{-2}(r)+o\left(r\dot{\varphi}(u)^{-2}(r)\right),
\label{control}
\end{align}
\end{theorem}
\begin{proof}
First of all, arguing as in  Theorem \ref{bowl}, Theorem \ref{existencecatenoid} and Proposition \ref{intervalodef} , \eqref{ab} has a unique ${\cal C}^\infty$-solution $u$ on $[r_0,\infty[$ which  is  strictly convex function satisfying that $\lim_{r\rightarrow\infty}u(r) = \infty$. Hence, from \eqref{equationrot},
\begin{equation}
\label{cp1}
r \, \dot{\varphi}(u)>u\,^\p, \qquad r\geq r_0.
\end{equation}
From Remark \ref{betapositive}, in order to study the asymptotic behavior of $\frac{u^\p}{\dot{\varphi}(u)}$, it is not a restriction to assume  that $\beta>0$. 

Take $\epsilon>0$ such that $\beta > 2 \varepsilon$, from \eqref{clinear} there exists $r_\varepsilon$ such that if $r\geq r_\varepsilon$,
\begin{equation}\label{desepsi}
-\varepsilon<\dot{\varphi}(u)(r) - \alpha\, u(r)-\beta<\varepsilon,  \quad -\varepsilon<\ddot{\varphi}(u)(r)-\alpha<\varepsilon.
\end{equation}
\begin{lemma}
\label{cl1} Consider for any $R>r_0$ the function, 
$$\zeta_R(r):=g_\varepsilon \left(u(R) + \int^{r}_{R}\,t \, \dot{\varphi}(u)(t)\, dt\right),  \quad r\geq R, \quad g_\varepsilon =\frac{\beta-2\varepsilon}{\beta+ \varepsilon}. $$ Then there exists $r_{1}\in \R$, depending only on $\varepsilon$, such that for any $R\geq r_1$, $ \zeta_{R}$ satisfies the following inequality,
\begin{equation}
\label{cp3}
\zeta_{R}^\pp<(1+\zeta_R^{\p {2}})\left(\dot{\varphi}(\zeta_R)-\frac{\zeta_R^\p}{r} \right), \qquad   r\geq R.
\end{equation}
\end{lemma}
\begin{proof} From the inequality \eqref{cp1}, $\zeta_R(r)>u(r)g_\varepsilon $. Hence, from \eqref{desepsi}, when $r$ is large enough, we have
\begin{equation}
\label{descp3}
\dot{\varphi}(\zeta_R)(r)>\alpha\, g_\varepsilon u(r)+\beta-\varepsilon.
\end{equation}
Using \eqref{cp1},  \eqref{desepsi} and by a straightforward computation,
\begin{equation}
\label{descp4}
\zeta_R^\pp(r)<g_\varepsilon \dot{\varphi}(u)(r)(1+(\alpha+\varepsilon) r^{2}), \qquad r\geq r_\varepsilon,
\end{equation}

On the other hand, from \eqref{descp3} and  \eqref{desepsi}, when $r\geq r_\varepsilon$, the following inequality holds,
\begin{equation}
\label{descp6}
(1+\zeta_R^{\p \,{2}})\left(\dot{\varphi}(\zeta_R)-\frac{\zeta_R^\p}{r} \right)>\varepsilon(1+\dot{\varphi} (u)^{2}\, r^{2}g_\varepsilon^2).
\end{equation}
Thus, \eqref{cp3} follows from \eqref{descp3}, \eqref{descp4}, \eqref{descp6} bearing in mind  that $u\rightarrow+\infty$ when $r\rightarrow +\infty$.
\end{proof}
\begin{lemma}\label{cl2} For any $R\geq r_0$ there exists $r_R\geq R$ such that $u^\p(r_R)-\zeta_R^\p(r_R)>0$. 
\end{lemma}
\begin{proof}
Assuming on the contrary,  if $ u^\p(r)-\zeta_R^\p(r)\leq0$ for any $r>R$, then the following inequalities holds,
$$\frac{u^\pp (r)}{1+u^{\p ^{2} }(r)}\geq \frac{3\varepsilon}{\beta + \varepsilon}\dot{\varphi}(u)(r)>\frac{3\varepsilon}{\beta + \varepsilon}\dot{\varphi}(u)(r_0),$$
Integrating, we can find a finite radius $\overline{r}$ such that $u'\rightarrow +\infty$ as $r\rightarrow \overline{r}$, getting a \textit{contraction} since the solution $u$ is defined for all $r>r_0$.
\end{proof}
Let's consider the function $d= u^\p - \zeta_R^\p$ on $[R,\infty[$. From Lemmas \ref{cl1} and \ref{cl2}, we can find $R\gg r_0$ verifying $u(R)>0$, $d(R)>0$ and such that  the inequality \eqref{cp3} holds. Hence, if
there exists a first $s\geq R$ such that $d(s)=0$ and $d^\p(s)<0$, we have
$$0>d^\p(s)=(1+u^\p(s)^{2})(\dot{\varphi}(u(s))-\dot{\varphi}(\zeta_R(s))).$$
On the other hand, as $d(r)>0$ for any $r\in]R,s[$ we have by integration of $d^\p$ that, 
$$u(s)>\zeta_R(s)+u(R)-\zeta_R(R)= \zeta_R(s) + \frac{3\varepsilon}{\beta +\varepsilon} u(R)>\zeta_R(s),$$
and  \eqref{conditions} gives that  $d^\p(s)>\dot{\varphi}(u(s))-\dot{\varphi}(\zeta_R(s))>0$ which is a \textit{contradiction}. 

Thus,  $d(r)>0$ for $r$  large enough and by using  the inequality \eqref{cp1}, we get,
\begin{equation}
\label{primeraaprox}
\frac{u^\p (r)}{\dot{\varphi}(u)(r)}=r+\mathcal{V}_1(r),  \qquad \text{with} \quad \lim_{r\rightarrow+\infty}\frac{\mathcal{V}_1(r)}{r}=0.
\end{equation}
Moreover, from the previous formula \eqref{primeraaprox} and  L'H\^{o}pital's rule, we also get that,
$$
\lim_{r\rightarrow +\infty} \frac{\log{\left(\dot{\varphi}^{2}(u(r))\right)}}{\alpha r^{2}}=1
$$
and $\dot{\varphi}(u)$ has the following asymptotic expansion,
\begin{equation}
\label{comporlog}
\dot{\varphi}(u)(r)=e^{\frac{1}{2}\alpha\, r^{2}+o(r^{2})}.
\end{equation}
\begin{lemma}\label{cl3}$\mathcal{V}_1\rightarrow 0$ as $r\rightarrow +\infty$.
\end{lemma}
\begin{proof} As  $\mathcal{V}_1$ is sublinear we have that for $r$ large enough,  $\vert\mathcal{V}_1(r)\vert<c\,r$ for all $c>0$.  Moreover,  from \eqref{ab} and the inequality \eqref{cp1}, $\mathcal{V}_1$ is a non-positive function and it satisfies the following differential equation,
\begin{equation}
\label{ecuacionxi}
\mathcal{V}_1^\p (r)=-\frac{\mathcal{V}_1(r)}{r}\left(1+\dot{\varphi}(u)^{2}(r)(r+\mathcal{V}_1(r))^{2} \right)-1-\ddot{\varphi}(u)(r)(r+\mathcal{V}_1(r))^{2}.
\end{equation}
Take   $\varepsilon>0$ and  $R\gg r_0$. If  $r\geq R$ and  $\mathcal{V}_1(r)\leq -\varepsilon$, from the sublinearity, we can suppose that $-r/2<\mathcal{V}_1(r)$ and,
\begin{equation}
\label{desigualdadessublinear}
\frac{r^{2}}{4}<(r+\mathcal{V}_1(r))^{2}<(c+1)^{2}r^{2}.
\end{equation}
Now, choosing $R$ large enough,   the equation \eqref{ecuacionxi} and the inequalities \eqref{desepsi} and \eqref{desigualdadessublinear} give,
\begin{equation}
\label{desxi}
\mathcal{V}_1^\p (r)\geq -1+\frac{\varepsilon}{r}+r\left(\frac{\varepsilon}{4}\dot{\varphi}(u)^{2}(r)-(\alpha+\varepsilon)(c+1)^{2}r \right).
\end{equation}
Using the conditions \eqref{conditions} and the asymptotic behavior  \eqref{comporlog}, $R$ may be  chosen large enough so that 
$$\dot{\varphi}(u)^{2}(r)\geq\frac{4}{\varepsilon}\left((\alpha+\varepsilon)(c+1)^{2}r+\frac{1}{r}\left(c+1-\frac{\varepsilon}{r}\right) \right),\quad r\geq R.$$
Thus, if $R$ is large enough and $r\geq R$ where  $\mathcal{V}_1(r)\leq -\varepsilon$, then  $\mathcal{V}_1^\p (r)\geq c>0$. Hence, $\mathcal{V}_1(r)\geq -\varepsilon$ for $r$ large enough and we conclude the proof.
\end{proof}
\begin{lemma}\label{cl4}  $\lim_{r\rightarrow +\infty}\frac{1}{ r}\dot{\varphi}^{2}(u)(r)\mathcal{V}_1(r) = -\alpha$. 
\end{lemma}
\begin{proof}  If $\lambda(r) = \frac{1}{r}\dot{\varphi}^{2}(u)(r)\mathcal{V}_1(r)$, then from \eqref{ab} and \eqref{primeraaprox} we have,  
\begin{align*}
\lambda^\p (r)=\dot{\varphi}^{2}(u)(r) &\left(2\mathcal{V}_1(r)\left(\ddot{\varphi}(u)(r)\left(1+\frac{\mathcal{V}_1(r)}{r}\right)-\frac{1}{r^2}\right)-\frac{1}{r}\right) \\
&+\dot{\varphi}^{2}(u)(r)\left(\frac{(r+\mathcal{V}_1(r))^{2}}{r}(-\ddot{\varphi}(u)(r)-\lambda(r))\right).
\end{align*}
Fix   $\varepsilon>0$ and  $R$ large enough.  Consider points $r\geq R$ where  $\lambda(r)\geq -\alpha+\epsilon$, then
\begin{equation}
\label{prides}
-\ddot{\varphi}(u)(r)-\lambda(r)\leq -\ddot{\varphi}(u)(r)+\alpha-\varepsilon
\end{equation}
and, if $R$ is large enough,  from  \eqref{clinear} and \eqref{prides},  we also get that,
$$-\ddot{\varphi}(u)(r)-\lambda(r)\leq -\frac{\varepsilon\alpha}{2}<0$$
and then  $\lambda^{\p}(r)< -1 $ when   $R$ is chosen sufficiently large. Hence, we obtain  that $\lambda(r)\leq -\alpha+\varepsilon$ for $r$ large enough.

In a similar way we may prove that $\lambda(r)\leq -\alpha-\varepsilon$ for $r$ sufficiently large.  
\end{proof}
Now \eqref{control} follows from \eqref{primeraaprox}, \eqref{comporlog} and Lemmas \ref{cl3} and \ref{cl4}.
\end{proof}

\begin{theorem}[{\bf Case $\alpha=0$}]
\label{alphazero}
Assume  that $\dot{\varphi}(u_0)\, r_0 \geq u_1$,  $\alpha=0$ and $\beta>0$.  Then the problem \eqref{ab} has an unique  strictly convex ${\cal C}^\infty$-solution $u$ on $[r_0,\infty[$. Moreover, if 
\begin{equation}
\label{condition2}
\lim_{u\rightarrow+\infty} u \, \ddot{\varphi}(u) = 0,
\end{equation}
we have the following asymptotic expansion:
\begin{equation}
\label{controlacero}
\frac{u^\p}{\dot{\varphi}(u)} (r)= r - \frac{1}{\beta^2\, r}+ o\left(r^{-1}\right),
\end{equation}
\end{theorem}

\begin{proof}
Arguing as in  Theorem \ref{bowl}, Theorem \ref{existencecatenoid} and Proposition \ref{intervalodef} , \eqref{ab} has a unique ${\cal C}^\infty$-solution $u$ on $[r_0,\infty[$ which  is  strictly convex function satisfying that $\lim_{r\rightarrow\infty}u(r) = \infty$. Moreover, as  Lemmas \ref{cl1} and \ref{cl2} also work in this case, we have the following asymptotic expansion
\begin{equation}\label{expresionphiacero}
\frac{u^\p}{\dot{\varphi}(u)}(r)=r+\mathcal{V}_1(r),
\end{equation}
where $\mathcal{V}_1$ verifies the same differential equation \eqref{ecuacionxi}, is also nonpositive and $\mathcal{V}_1(r)\rightarrow 0$. Moreover, from  \eqref{clinear}, $\dot{\varphi}$ writes as 
\begin{equation}
\dot{\varphi}(u)(r)=\beta+o(1). 
\end{equation}
Consider now  the new function $\mathcal{V}_2(r)=r\, \dot{\varphi}^{2}(u)(r)\mathcal{V}_1(r)$. Then
\begin{align*}
\mathcal{V}_2^\p =r \, \dot{\varphi}^{2} \left(2\ddot{\varphi}\mathcal{V}_1(r+\mathcal{V}_1)-1 +\frac{(r+\mathcal{V}_1)^2}{r^2}(-r^{2}\, \ddot{\varphi}-\mathcal{V}_2)\right).
\end{align*}
From the expressions \eqref{condition2}, \eqref{expresionphiacero}  and L'H\^{o}pital's rule, we have
\begin{equation}
\label{limitesacero}
\lim_{r\rightarrow +\infty}\ddot{\varphi}(u(r))\, r=0 \emph{ and } \lim_{r\rightarrow +\infty}\ddot{\varphi}(u(r))\, r^{2}=0,
\end{equation}
and working as in Lemma \ref{cl4} we can prove  that $\mathcal{V}_2(r)\rightarrow -1$. 
Finally, the Theorem follows from the expansion \eqref{expresionphiacero} as $r\rightarrow +\infty$.
\end{proof}
\subsection{{\sc Proof of Theorem A}}
If  $\alpha>0$,  from \eqref{primeraaprox} and \eqref{comporlog} we can write,
\begin{equation}
\label{derUpsilonexpli}
\log (\dot{\varphi}(u)) (r) =  \frac{\alpha r^2}{2} +\Upsilon(r),
\end{equation}
where $\Upsilon^\p = (\ddot{\varphi}-\alpha) r + \ddot{\varphi}  {\cal V}_1$. Hence, as the  first non-vanishing $a_k$ is positive, for $r$ large enough $\Upsilon$ is a decreasing function in $r$ such that $-\infty< c= \lim_{r \rightarrow +\infty} \Upsilon(r)$ otherwise from Lemma \ref{cl4},  \eqref{series2},  \eqref{derUpsilonexpli} and by using L'H\^{o}pital's rule, we have that,
\begin{align*} \label{limiteupexpli}
+\infty&=\lim_{r\rightarrow+\infty}\dot{\varphi}^2(u)(r) = \lim_{r\rightarrow+\infty}\frac{e^{2\Upsilon}}{e^{-\alpha r^{2}}}
=\lim_{r\rightarrow+\infty}\frac{\left(e^{2\Upsilon}\right)^\p}{\left(e^{-\alpha r^{2}}\right)^\p}\\
&=-\lim_{r\rightarrow+\infty}\frac{\dot{\varphi}(u)^2(r) \left((\ddot{\varphi}(u)(r)-\alpha) r + \ddot{\varphi}(u)(r)  {\cal V}_1(r)\right)}{\alpha r}=\alpha \, a_1,
\end{align*}
which is a contradiction. 

Applying again L'H\^{o}pital's rule to $\lim_{r\rightarrow+\infty}\frac{e^{2\Upsilon}- e^{2 c}}{e^{-\alpha r^{2}}}$, we have   
$$ \dot{\varphi}^2(u)(r)  = e^{\alpha r^2 + 2c} + O(1) \quad \text{and}\quad  \lim_{r\rightarrow+\infty}O(1)=\alpha \, a_1.$$
Thus, from  Lemma \ref{cl4} and Theorem \ref{comportamientoasin}
$$ \varphi(u)^\p(r)= r  e^{\alpha r^2 + 2c} + \alpha a_1 r + o(r) ,$$
and  \eqref{applineal} follows by integration of the above  expression.

\

If $\alpha=0$ then, the condition \eqref{condition2} follows from  \eqref{expresionphiacero} and we have that
\begin{equation}
\frac{u^\p}{\dot{\varphi}(u)} (r)= r - \frac{1}{\beta^2\, r} + o\left(r^{-1}\right).
\end{equation}
Now, by taking  $ \mathcal{V}_3(r) = (\mathcal{V}_2(r) +1) r^2 $ we get
\begin{align*}
\mathcal{V}_3^\p &= \frac{2\mathcal{V}_3}{r} +  r ^3\, \dot{\varphi}^{2} \left(2\ddot{\varphi}\mathcal{V}_1(r+\mathcal{V}_1)-1 +\frac{(r+\mathcal{V}_1)^2}{r^2}(-r^{2}\, \ddot{\varphi}+ 1-\frac{\mathcal{V}_3}{r^2})\right)\\
&=r \dot{\varphi}^{2} \left(\frac{2\mathcal{V}_3}{\dot{\varphi}^{2} r^2} +  2r^4 \ddot{\varphi}\frac{\mathcal{V}_1(r+\mathcal{V}_1)}{r^2}-r^2 +\frac{(r+\mathcal{V}_1)^2}{r^2}(-r^{4}\, \ddot{\varphi}+ r^2-\mathcal{V}_3)\right)\\
&=r \dot{\varphi}^{2} \frac{(r+\mathcal{V}_1)^2}{r^2}\left(-r^{4}\ddot{\varphi}+ r^2\left(1-\frac{r^2}{(r+\mathcal{V}_1)^2}\right)-\mathcal{V}_3\right) \\
&+ r \dot{\varphi}^{2} \left(\frac{2\mathcal{V}_3}{\dot{\varphi}^{2} r^2} +  2r^4 \ddot{\varphi}\frac{\mathcal{V}_1(r+\mathcal{V}_1)}{r^2}\right).
\end{align*}
But, from   \eqref{expresionphiacero} and L'H\^{o}pital's rule, we obtain
\begin{align*}
&\lim_{r\rightarrow +\infty}\ddot{\varphi}(u(r))\, r^4=-\frac{4 a_1}{\beta^2},\\
&\lim_{r\rightarrow +\infty}r^2\left(1-\frac{r^2}{(r+\mathcal{V}_1)^2}\right) = -\frac{2}{\beta^2}
\end{align*}
thus, by working as in Lemma \ref{cl4} we prove that 
$$\lim_{r\rightarrow\infty}\mathcal{V}_3(r) = \frac{-2 + 4a_1}{\beta^2}.$$
Hence,
\begin{equation*}
\frac{u^\p}{\dot{\varphi}(u)} (r)= r - \frac{1}{\beta^2\, r}-  \frac{2 -4 a_1}{\beta^4\, r^3} + o\left(r^{-3}\right),
\end{equation*}
and \eqref{casoalphacero} follows from integration in the above expression.
\section{Uniqueness of  globally convex solutions}
Along this section  $\varphi:]a,+\infty[\longrightarrow \R$ will be  a regular  function satisfying the expansion \eqref{series2}.

\

\noindent For any $\theta\in [0,2\pi[$ we consider $\vec{v}=(\cos\theta,\sin\theta,0)$ and denote by $\Pi_{\vec{v}}(t)$ the vertical plane
\begin{equation}\label{plane}
\Pi_{\vec{v}}(t) = \{ p\in \R^3 \, | \, \langle p,\vec{v}\rangle  = t\}
\end{equation}
\begin{definition}{\rm
Let $\Sigma_1$ and $\Sigma_2$ be two arbitrary subsets of $\mathbb{R}^{3}$. We say that $\Sigma_1$ is on the right hand side of $\Sigma_2$ respect to  $\Pi_{\vec{v}}(t)$ and write $\Sigma_1\ge_{\vec{v}}\Sigma_2$ if and only if for every point $q\in\Pi_{\vec{v}}(t)$ such that,
$$\pi^{-1}(q)\cap \Sigma_1\neq\emptyset \quad \text{and} \quad  \pi^{-1}(q)\cap \Sigma_2\neq\emptyset,$$
we have the following inequality, 
$$\inf \{\langle p,\vec{v}\rangle \,:\,  p\in \pi^{-1}(q)\cap \Sigma_1 \} \geq \sup\{\langle p,\vec{v}\rangle \,:\, p\in\pi^{-1}(q)\cap \Sigma_2\},$$
where $\pi:\mathbb{R}^{3}\rightarrow\Pi_{\vec{v}}(t)$ denotes the orthogonal projection on $\Pi_{\vec{v}}(t)$.}
\end{definition}

For an arbitrary subset $M$ of $\mathbb{R}^{3}$ we also consider the following subsets:
\begin{align*}
&\Sigma_+(t):=\{p\in M \,:\,\langle p,\vec{v}\rangle \, \geq t\}.\\
&\Sigma_-(t):=\{p\in M\,:\,\langle p,\vec{v}\rangle \, \leq t\}.\\
&\Sigma_+^{*}(t):=\{p + 2(t-\langle p,\vec{v}\rangle )\vec{v}\in\mathbb{R}^{3}\,:\, p\in \Sigma_+(t)\}.\\
&\Sigma_-^{*}(t):=\{p + 2(t-\langle p,\vec{v}\rangle )\vec{v}\in\mathbb{R}^{3}\,:\, p\in \Sigma_-(t)\}.
\end{align*}
From Theorem A, it is natural to study $[\varphi,\vec{e}_{3}]$-minimal surfaces whose behavior at infinity is of rotational type. To be more precise, 
\begin{definition}{\rm 
We say that a  $[\varphi,\vec{e}_{3}]$-minimal end  $\Sigma$  is {\sl smoothly asymptotic} to a  rotational-type example if  $\Sigma$ can be expressed outside a ball as a vertical graph of  a function $u_\Sigma$ so that, according to $\alpha$  is either positive or zero,  one of the following expressions holds
\begin{equation}\label{defasym}
\varphi (u_{\Sigma})(x)= C \,e^{\alpha \,|x|^2}  + O\left(|x|^2\right),\quad \text{if}\quad \alpha>0,
\end{equation}
where $C$ is a positive constant or up to a constant,
\begin{equation}\label{defasym2}
{\cal G}(u_{\Sigma})(x)= \frac{|x|^2}{2}-\frac{1}{\beta^2}\log(|x|)+ {O}\left(|x|^{-2}\right),
\end{equation}
if $\alpha =0$ and $\beta>0$.}
\end{definition}
Let $\Sigma$ be an embedded  $[\varphi,\vec{e}_{3}]$-minimal surface  $\Sigma$  with a single end smoothly asymptotic to a  bowl-type example. Then, there exists $R>0$ large enough such that $\Sigma\cap(\mathbb{R}^{3}\backslash B(0,R))$ is the vertical graph of a function $u_{\Sigma}$ verifying either  \eqref{defasym} if $\alpha>0$ or \eqref{defasym2} if $\alpha=0$ and $\beta>0$.
\begin{lemma}\label{lema1}
There exists $r_1>R$  such that if $t>r_1$ then $\Sigma_+(t) $ is a graph over $\Pi_{\vec{v}}(t)$.
\end{lemma}
\begin{proof}
It is clear that when $t>R$,  $\Sigma_+(t)$ has only one component which is unbounded. Moreover, if $\alpha>0$ then from \eqref{defasym}, 
$$ \dot\varphi(u_\Sigma)(x)( du_\Sigma)_x(\vec{v}) \geq  2  \alpha\, e^{\alpha \,|x|^2} \langle x,\vec{v}\rangle \left( C +e^{-\alpha \,|x|^2} g(|x|)\right),$$
where $$\lim_{|x|\rightarrow}\frac{g(|x|)}{|x|^2}=0.$$ Hence, there exists $r_1$ large enough such that if $\langle x,\vec{v}\rangle \, \geq r_1$, then $(du_\Sigma)_x(\vec{v})>0$ and, in this case,  the Lemma follows because $\Sigma$ is embedded and $\Sigma_+(r_1) \cup \pi(\Sigma_+(r_1))$ bounds a domain in $\R^3$.

When $\alpha=0$ a similar argument with \eqref{defasym2} also works.
\end{proof} 

From Lemma \ref{lema1}, fixed $t>r_1$, $\Sigma^*_+(t)\cap \{p\in \R^3 : \langle p,\vec{e}_3\rangle> R\}$ is the vertical graph of the function   satisfying 
\begin{equation}\label{ustar}
u^{*}_{t}(x)=u_\Sigma(x + 2(t-\langle x,\vec{v}\rangle )\vec{v})
\end{equation}
\begin{lemma}\label{lema2}
Consider  $a>0$ not depending on $R$ and $\epsilon_0>0$. Then, for $R$ large enough and  $t> a + \langle x,\vec{v}\rangle $,  we have
$$ u^{*}_{t}(x) - u_\Sigma(x) >\epsilon_0>0.$$
\end{lemma}
\begin{proof}
If $\alpha >0$ then, from \eqref{defasym} and \eqref{ustar}, we obtain
\begin{align*}
\varphi(u_t^*)(x) - \varphi(u_\Sigma)(x) &\geq C\, e^{\alpha \, |x|^2}\left( e^{4\alpha t (t - \langle x,\vec{v}\rangle )} -1\right)\\
& - M \left( 2|x|^2 + 4 t (t - \langle x,\vec{v}\rangle )\right),
\end{align*}
for some positive constant $M$. Hence, taking $\lambda$ such that $$ \frac{1+\sqrt{1+\lambda}}{\lambda}< \frac{R}{2 t}$$ and $R>\alpha^{-1}$, 
we have that $ 4 t (t - \langle x,\vec{v}\rangle )\leq \lambda |x|^2$ and 
\begin{align*}
\varphi(u_t^*)(x) - \varphi(u_\Sigma)(x) &> C\, e^{\alpha \, R^2}\left( e^{4\alpha R\, a} -1- M e^{-\alpha \,R^2}(\lambda + 2) R^2\right)>0
\end{align*}
for $R$ large enough. The  result follows because $\varphi$ is strictly increasing.

\

When $\alpha =0$, we can estimate ${\cal G}(u_t^*)(x)-  {\cal G}(u_\Sigma)(x)$  as in \cite[Claim 1, Step 3]{MSHS2}  and to use that ${\cal G}$ is a strictly increasing function.
\end{proof}
\subsection{{\sc Proof of Theorem B}}
The main idea is to use the Alexandrov's reflection principle, \cite{Al},  for proving  that $\Sigma$ is symmetrical with respect to $\Pi_{\vec{v}}(0)$.  For proving that, it is not difficult to  see that  Lemma \ref{lema1} and Lemma \ref{lema2} are the fundamental facts we need to check that all the steps in the proof of Theorem A  in \cite{MSHS2} can be adapted to our case and for getting  to prove that $0\in\mathcal{A}$ were 
$$\mathcal{A}:=\{t\geq 0:\Sigma_{+}(t)  \emph{is a graph over } \Pi_{\vec{v}}(t) \emph{ and } \Sigma_{+}^{*}(t)\ge_{\vec{v}}\Sigma_{-}(t)\}.$$
 A symmetrical argument gives that $\Sigma_{-}^{*}(0)\leq_{\vec{v}} \Sigma_{+}(0)$. Hence, $\Sigma_{+}^{*}(0)=\Sigma_{-}(0)$ and  $\Sigma$ is symmetric respect to the plane $\Pi_{\vec{v}}(0)$. As $\vec{v}=(\cos\theta, \sin\theta, 0)$ represents any unit horizontal vector,  $\Sigma$ would be a revolution surface  touching the axis of revolution, that is, a $[\varphi,\vec{e}_3]$-minimal bowl.
 \begin{remark}
 It would be interesting to give a similar results for $[\varphi,\vec{e}_{3}]$-maximal surfaces in the Lorentz-Minkowski space $\mathbb{L}^{3}$ using the Calabi's Type correspondence of {\rm \cite{MM}}.
 \end{remark}

\end{document}